%% file: actual_paper.tex
\documentclass[12pt]{article}

\usepackage{tikz}
\usepackage{tikz-cd}
\usepackage{amsmath}
\usepackage{amsthm}
\usepackage{amssymb}
\usepackage{geometry}
\usepackage{subcaption}
\usepackage[font=normalsize]{caption}

\newtheorem{tm}{Theorem}[section]
\newtheorem{df}[tm]{Definition}
\newtheorem{lm}[tm]{Lemma}
\newtheorem{pp}[tm]{Proposition}
\newtheorem{cor}[tm]{Corollary}

\newtheorem{examp}[tm]{Example}
\newtheorem{rmk}[tm]{Remark}

\begin{document}

\title{A Generating Polynomial for Convex Subsets of Posets}

\author{Ian George \and Rowan Yeats\thanks{RY is partially supported by an NSERC Discovery grant and the Canada Research Chairs program.}}

\maketitle

\abstract{
For any finite poset we define a generating polynomial counting upsets, downsets, and their intersection.  We investigate the behaviour of this polynomial with respect to poset operations and show that it distinguishes certain posets constructed with the operations of ordinal sum and disjoint union, including Series-Parallel posets as a special case. Lastly, we comment on connections to the causal set approach to quantum gravity.
}

\section{Introduction}

In this paper we introduce and investigate a generating polynomial for convex sets of a finite poset.  Enumerating upsets/downsets and intervals in posets is a problem that has garnered attention in the recent literature.  For many nice families of posets, such enumeration involves rich combinatorics, e.g. \cite{bousquet2023intervals, azam2023generating, baril2025enumeration, chapoton2006nombre, preville2017enumeration, melou18number, baril2024intervals}.  A convex subset of a poset is the intersection of an upset and downset and thus is a generalization of an interval, that being the case where the upset and downset are each generated by a single element.

Our polynomial is a polynomial with five variables which records convex set size, the sizes of the upset and downset intersecting to give the convex set, and the sizes of the antichains that generate the upset and downset.  In this paper we outline the basic properties of the polynomial and the information it contains.  We give formulas for the polynomial under the operations of disjoint union and ordinal sum, and show that it distinguishes posets in certain families of posets constructed iteratively with these operations.  As a special case, the polynomial distinguishes Series-Parallel posets.

This polynomial is motivated from the causal set approach to quantum gravity.  Previous work has shown that the intervals of a causal set, or locally-finite poset, contain much information of physical significance. We aim to describe ways in which considering all convex sets may extend these results.  There has also been interest in the causal set community in considering the sizes of the upset and downset generated by each element of the poset. Our polynomial carries this information as well and so we also generalize this construction.  

\subsection{Introduction to Generating Functions}

Since the aim of this paper is for it to reach a wider audience including physicists, we include here a short introduction to basic concepts in enumerative combinatorics.  A standard reference for this topic is \cite{wilf2014generatingfunctionology}.  Let $S$ be a (possibly infinite) set.  A function $\omega : S \to \mathbb{N}$ is a \emph{weight function} if for all $n \in \mathbb{N}$, the set $\omega^{-1}(n)$ has finite cardinality.  Given such a weight function, we define the \emph{(ordinary) generating function} of $\omega$ to be the formal power series

\[ f_\omega(z) = \sum_{s \in S} z^{\omega(s)} = \sum_{n \in \mathbb{N}} f_n z^n \]

\noindent where $f_n = |\omega^{-1}(n)|$, and $z$ is an indeterminate.  The idea here is that $S$ is typically some set of discrete objects that we want to count, the weight function is whatever notion of size or weight we want to count with respect to, and then $f_n$ is the number of these objects of weight $n$, which is the thing we want to understand better. The generating function is an algebraic gadget to carry and help us understand the series of coefficients, that is, to help us understand the series of the counts of how many of our objects there are of each weight.  It will be useful to have notation for the operation of \emph{coefficient extraction} from a generating function: $[z^k]f_\omega(z)$ is defined to be the coefficient of $z^k$ in $f_\omega(z)$, that is, $f_k$.  In other words, $[z^k]f_\omega(z) = \frac{1}{k!}\frac{\partial^k}{\partial z^k}f_\omega(z)|_{z=0}$, where the derivative is a formal derivative with respect to $z$.

The ideas above are naturally generalized to the multivariate case.  A multivariate weight function is a function $\omega : S \to \mathbb{N} \times \cdots \times \mathbb{N} = \mathbb{N}^m$ where there are $m$ factors of $\mathbb{N}$, and each $\omega^{-1}((n_1, \ldots, n_m))$ has finite cardinality.  Thus we represent $\omega(s)$ by a vector $(\omega_1(s), \ldots, \omega_m(s))$ where $s \in S$ and $\omega_i$ is an element of the $i$th factor of $\mathbb{N}$.  If $\mathbf{z} = (z_1, \ldots, z_m)$ is a vector of indeterminates, then we define $\mathbf{z}^{\omega(s)} = \prod_{i=1}^m z_i^{\omega_i(s)}$.  The multivariate generating function for $\omega$ is thus

\[ f_\omega(\mathbf{z}) = \sum_{s \in S} \mathbf{z}^{\omega(s)} = \sum_{\mathbf{n} = (n_1, \ldots, n_m) \in \mathbb{N}^m} f_\mathbf{n} \mathbf{z}^\mathbf{n} \]

\noindent where $f_\mathbf{n} = |\omega^{-1}(\mathbf{n})|$.  Coefficient extraction in the multivariate case is a straight-forward extension from the univariate case, as if $\mathbf{k} = (k_1, \ldots, k_m) \in \mathbb{N}^m$ then $[\mathbf{z}^\mathbf{k}]f_\omega(\mathbf{z}) = f_\mathbf{k} = \frac{1}{k_1! \cdots k_m!}\frac{\partial^{k_1}}{\partial z_1^{k_1}} \cdots \frac{\partial^{k_m}}{\partial z_m^{k_m}}f_\omega(\mathbf{z})$.

Generating functions become a powerful tool when natural operations on our discrete objects correspond to algebraic operations on the generating functions.  In that situation, simply understanding how the objects decompose under such nice operations gives equations for the generating functions which can be solved or otherwise analysed in order to extract information about the counts of our objects.  In nice cases, this can give closed form solutions for the counts which may be unobtainable by elementary direct counting arguments, and even when closed forms are not possible, information such as when the objects are equinumerous with other interesting classes or objects, or asymptotics for the counts can often be extracted.

\section{The Polynomial}

\subsection{Definitions}

All posets considered here will be finite, so we will simply say poset instead of finite poset.  We follow \cite{stanley2011enumerative} for standard poset definitions, but we will briefly give the  main definitions here for completeness and to set notation.  Let $P$ be a poset with partial order $\leq$.  For $u, v \in P$, the \emph{interval} $[u, v]$ is the set $\{w \in P \;:\; u \leq w \leq v\}$.  A set $U \subseteq P$ is \emph{upward-closed}, or is an \emph{upset}, if for each $u \in U$, if $v \in P$ and $u \leq v$, then $v \in U$.  Likewise a set $D \subseteq P$ is \emph{downward-closed}, or is a \emph{downset} (sometimes also called an \emph{ideal}), if for each $u \in D$, if $v \in P$ and $v \leq u$, then $v\in D$.  For a set $S \subseteq P$ we say the \emph{upset generated by $S$} is $V(S) = \{v \in P \; : \; u \leq v \; {\rm for \; some} \; u \in S\}$, and similarly the \emph{downset generated by $S$} is $\Lambda(S) = \{u \in P \;:\; u \leq v \; {\rm for \; some} \; v \in S\}$.  Indeed, $V(S)$ is an upset  of $P$ and $\Lambda(S)$ is a downset of $P$.  When $S = \{u\}$, we write $V(u)$ instead of $V(\{u\})$ and $\Lambda(u)$ instead of $\Lambda(\{u\})$ for the \emph{principal upset} and \emph{principal downset} generated by $u$.  Given an upset $U \subseteq P$, there is a unique antichain $A = A(U)$ such that $U = V(A)$.  Likewise for a downset $D \subseteq P$ there is a unique antichain $A = A(D)$ such that $D =\Lambda(A)$.  For a subset $S \subseteq P$, we let $\min(S)$ and $\max(S)$ be the minimal and maximal elements of $S$, respectively.

A subset $X \subseteq P$ is \emph{convex} if $u, v \in X$ implies $[u, v] \subseteq X$.  This is a natural poset-theoretic analogue of the typical notion of convexity.  Given convex $X \subseteq P$, there exists a unique upset $U$ and downset $D$ such that $X = U \cap D$, $A(U) = \min(X) \subseteq D$, and $A(D) = \max(X) \subseteq U$.  In this way, convex sets generalize intervals, since $[u, v] = V(u) \cap \Lambda(v)$.  In the literature, convex sets are sometimes referred to as \emph{interval-closed} sets \cite{elder2023toggling, lafreniere2025interval, elizalde2024enumeration} and we previously, in an extended abstract version of these results, called them \emph{generalized intervals} \cite{george2025generalized}.  Convex sets of posets have received some attention in the combinatorics literature, from investigating the lattice of convex sets of a poset \cite{birkhoff1985convexity} to convex sets in products of chains and rowmotion on convex sets \cite{elder2023toggling, elizalde2024enumeration, lafreniere2025interval}.

Let $\mathcal{X}(P)$ be the set of convex sets of $P$.  Consider the weight function $\omega : \mathcal{X}(P) \to \mathbb{N}^5$ given by $\omega(X) = (|\min(X)|, |\max(X)|, |\Lambda(\max(X))|, |V(\min(X))|, |X|)$.  This is indeed a weight function, since $P$ is finite.  Our generating polynomial is the generating polynomial of $\omega$:

\begin{df}
    We define a generating polynomial $\Phi(P ; s, t, x, y, z)$ for convex subsets of a poset $P$ as follows:

    \[ \Phi(P ; s, t, x, y, z) = \sum_{\substack{X \subseteq P \\ X\;\text{convex}}} s^{|\min(X)|} t^{|\max(X)|} x^{|\Lambda(\max(X))|} y^{|V(\min(X))|} z^{|X|}. \]

\end{df}

\noindent Note that this polynomial will have a constant term of $1$ corresponding to $X = \emptyset$.  This is the only convex set for which the exponent on any of $s, t, x, y, z$ is zero, hence every non-constant term is divisible by $stxyz$. 

\begin{examp}\label{eg single elt}
Writing $\bullet$ for the single element poset we have $\Phi(\bullet; s, t, x, y, z) = 1+stxyz$ since $\bullet$ has exactly two convex subsets, the empty set and the entire poset.
\end{examp}

\begin{examp}
    Consider the poset $P$ in Figure \ref{fig:pos_diamond} with elements $\{a, b, c, d\}$ and non-trivial relations $a \leq b \leq d$ and $a \leq c \leq d$.  The convex set generating polynomial of $P$ is

    \begin{align*}
        \Phi(P ; s, t, x, y, z) = 1+ stxyz(x^3y^3z^3 & + sx^3y^2z^2 +  
        tx^2y^3z^2 + stx^2y^2z \\
        & \quad + 2x^3yz + 2xy^3z + x^3 + y^3 + 2xy).
    \end{align*}
    
    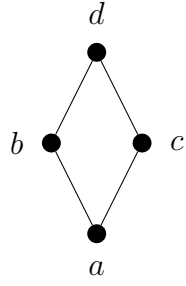
\begin{figure}
        \centering
        \input{figs/poly_example}
        \caption{Diamond poset.}
    \label{fig:pos_diamond}
    \end{figure}
\end{examp}

\subsection{Polynomial Specifications/Combinatorial Information}

We can extract combinatorial information about $P$ by extracting coefficients and taking certain evaluations of $\Phi(P ; s, t, x, y, z)$, as summarized in Table \ref{tab:info_from_poly}.  For example, the number of relations is $[st]\Phi(P ; s, t, 1, 1, 1)$ since a relation corresponds precisely with non-empty intervals, which are convex sets with a unique minimal and maximal element.  Setting $x=y=z=1$ simply means we seek the total number of relations, and ignore the sizes of intervals/upsets/downsets that they form.

\begin{table}[h]
    \centering
    \begin{tabular}{|c||c|}
    \hline
    Poset information/property & How to obtain it \\
    \hline
    Number of elements & \begin{tabular}{@{}c@{}}$[stz]\Phi(P ; s, t, 1, 1, z)$, or maximum $n$ \\ such that $[z^n]\Phi(P ; s, t, x, y, z) \neq 0$\end{tabular} \\
    \hline
    Number of cover relations & $[stz^2]\Phi(P ; s, t, 1, 1, z)$ \\
    \hline
    Number of relations & $[s t] \Phi(P ; s, t, 1, 1, 1)$ \\
    \hline
    Number of minimal elements & $m$ such that $[s^m y^{|P|}]\Phi(P ; s, t, x, y, z) \neq 0$ \\
    \hline
    Number of maximal elements & $m$ such that $[t^m x^{|P|}]\Phi(P ; s, t, x, y, z) \neq 0$ \\
    \hline
    Number of antichains of size $k$ & $[s^k t^k z^k]\Phi(P ; s, t, 1, 1, z)$ \\
    \hline
    Poset width & Maximum $k$ such that $[s^kt^kz^k]\Phi(P ; s, t, x, y, z) \neq 0$ \\
    \hline
    \end{tabular}
    \caption{Extracting information about $P$ from the polynomial.}
    \label{tab:info_from_poly}
\end{table}

In addition to the above information, we can specialize $\Phi(P; s, t, x, y, z)$ to other generating polynomials.  In Table \ref{tab:other_gens_from_poly} we summarize some other generating polynomials that can be obtained from $\Phi$.  Note that in any term of the polynomial the exponent of $x$ and $y$ is greater than or equal to the exponent of $z$, and the exponent of $z$ is greater than or equal to the exponent of $s$ and $t$.  Thus for any of the substitutions involving $w^{-1}$, the result is still a polynomial and so we may set $w=0$.  For example, the generating function of upsets (by cardinality) is $\Phi(P ; 1, 1, 1, yw, 1/w)|_{w=0}$ since an upset is precisely a convex set $X$ for which $|X| = |V(\min(X))|$.  Thus upsets correspond to terms of $\Phi(P ; 1, 1, 1, yw, 1/w)$ for which the exponent of $w$ is zero.  Since the exponent of $y$ in such terms is $|X|$, the result follows.

The generating polynomial for downsets appears as a special case in \cite{proctor1984bruhat}.  Here, the author defined a generating polynomial for \emph{flags} of downsets, that is, sequences of downsets totally ordered by inclusion.  Of particular interest was the case of minuscule posets.  The generating polynomial of downsets is simply the special case of considering flags of length one.  In \cite{ding2019antichain}, the authors studied the antichain generating polynomials for Cartesian products of finite, graded posets with chains.  They also considered the case when the non-chain factor is a minuscule poset, showing that these antichain generating polynomials are related to certain Narayana polynomials.  Lastly, the generating polynomials for principal downsets/upsets (or both) is of possible interest from the causal set perspective \cite{sorkin_comm}: these polynomials are polynomial forms of so-called pf-invariants, which have been investigated (unpublished) for their causal set distinguishing properties.

\begin{table}[h]
    \centering
    \begin{tabular}{|c||c||c|}
    \hline
    Generating polynomial for sizes of... & How to obtain it & Found in \\
    \hline
    Convex sets & $\Phi(P ; 1, 1, 1, 1, z)$ & \\
    \hline
    Upsets, downsets, convex sets & $\Phi(P ; 1, 1, x, y, z)$ & \\
    \hline
    Upsets & $\Phi(P ; 1, 1, 1, yw, 1/w)|_{w=0}$ & \\
    \hline
    Downsets & $\Phi(P ; 1, 1, xw, 1, 1/w)|_{w=0}$ & \cite{ding2019antichain}, \cite{proctor1984bruhat} \\
    \hline
    Principal upsets & $[s]\Phi(P ; s, 1, 1, yw, 1/w)|_{w=0}$ & \cite{sorkin_comm} \\
    \hline
    Principal downsets & $[t]\Phi(P ; 1, t, 1, yw, 1/w)|_{w=0}$ & \cite{sorkin_comm} \\
    \hline
    Principal upsets and downsets & $[z] \Phi(P ; 1, 1, x, y, z)$ & \cite{sorkin_comm} \\
    \hline
    Intervals & $[st] \Phi(P ; s, t, 1, 1, z)$ & \\
    \hline
    Principal upsets, downsets, intervals & $[st] \Phi(P ; s, t, x, y, z)$ & \\
    \hline
    Antichains & $\Phi(P ; 1/w, 1, 1, 1, zw)|_{w=0}$ & \cite{ding2019antichain, chen2026antichain} \\
    \hline
    \end{tabular}
    \caption{Other generating polynomials obtained from the polynomial.}
    \label{tab:other_gens_from_poly}
\end{table}

\subsection{Distinguishability}

We saw in the previous section some of the combinatorial information that is captured in our generating polynomial.  We now ask the natural question: exactly how much can we determine about a poset $P$ just from its polynomial? 
Computations in SageMath \cite{sagemath} reveal that $P$ is not determined (up to isomorphism) by its convex set polynomial. 
At six elements, the two posets of Figure \ref{fig:6_elem_non_iso} are the smallest two posets that are indistinguishable by their polynomials.  For seven elements there are 16 pairs of posets which are not distinguished, while for eight elements there 249 pairs.  Appendix \ref{app:A} lists the 16 pairs with seven elements.  

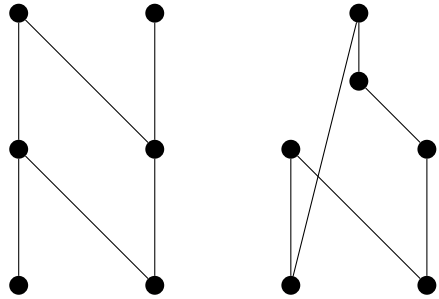
\begin{figure}[h]
    \centering
    \input{figs/6elem_not_dist}
    \caption{The two posets with six elements having equal convex set generating polynomial.}
    \label{fig:6_elem_non_iso}
\end{figure}

We also remark on the distinguishing power of the lattice of convex sets explored in \cite{birkhoff1985convexity}.  The authors gave necessary and sufficient conditions for two posets to have isomorphic lattices of convex sets.  From these conditions, which we will not give the details of here, it follows that the posets of Figure \ref{fig:6_elem_non_iso} have non-isomorphic lattices of convex sets.  However, any poset and its dual have isomorphic lattices of convex sets.  Thus the poset on the right of Figure \ref{fig:6_elem_non_iso}, for example, is not distinguished from its dual by the lattice of convex sets while it is distinguished by the convex set generating polynomial.  Therefore, these two objects possess distinct distinguishing capabilities.

\subsection{Poset Operations}

As $\Phi(P ; s, t, x, y, z)$ is a generating polynomial encoding combinatorial information about $P$, it is of interest to know how it behaves with poset operations.  Given a poset $P$, its \emph{dual} $P^*$ is the poset with the same underlying set but with relations $u \leq_{P^*} v$ if and only if $v \leq_P u$.  It is immediate from this definition that $S$ is an upset of $P$ if and only if $S$ is a downset of $P^*$, and $A$ is an antichain of $P$ if and only if it is an antichain of $P^*$.  Thus we find that $\Phi(P^* ; s, t, x, y, z) = \Phi(P ; t, s, y, x, z)$.

Given two posets $P, Q$ on disjoint underlying sets, their \emph{disjoint union} $P+Q$ is the poset with underlying set $P \cup Q$ and relations $u \leq_{P + Q} v$ if and only if $u, v \in P$ and $u \leq_P v$ or $u, v \in Q$ and $u \leq_Q v$.

\begin{pp}
    $\Phi(P+Q ; s, t, x, y, z) = \Phi(P ; s, t, x, y, z) \Phi(Q ; s, t, x, y, z)$.
    \label{prop:union}
\end{pp}

\begin{proof}
    This follows since every convex set of $P+Q$ is either empty, a subset of $P$, a subset of $Q$, or the disjoint union of a convex set of $P$ and a convex set of $Q$.
\end{proof}

The \emph{ordinal sum}, $P \oplus Q$, of $P$ and $Q$ also has underlying set $P \cup Q$, but has the relations of $P+Q$ as well as the relations $u \leq_{P \oplus Q} v$ if $u \in P$, $v \in Q$.  The cover relations of $P \oplus Q$ are precisely the cover relations of $P$ and $Q$ along with the relations $u \leq v$ such that $u$ is a maximal element of $P$ and $v$ a minimal element of $Q$.

\begin{pp}
    \begin{align*}
        & \Phi(P \oplus Q ; s, t, x, y, z) \\& = y^{|Q|}(\Phi(P; s, t, x, y, z)-1) + x^{|P|}(\Phi(Q; s, t, x, y, z)-1) \\
        & \quad + x^{|P|} y^{|Q|} (\Phi(P ; s, 1, 1, yw, z/w)|_{w=0}-1)(\Phi(Q ; 1, t, xw, 1, z/w)|_{w=0}-1) + 1.
    \end{align*}
    \label{pp:ordinal_sum}
\end{pp}

\begin{proof}
    There are four cases for a convex set $X \subseteq P \oplus Q$.  The first case is $X = \emptyset$, corresponding to the constant $1$.  The terms for $X \neq \emptyset$ and $X \subseteq P$ correspond to terms of $\Phi(P ; s, t, x, y, z) - 1$ multiplied by $y^{|Q|}$ since every upset $U$ in $P \oplus Q$ generated by elements in $P$ can be written as $U = (U \cap P) \cup Q$, and conversely.  A similar argument holds for convex $X \subseteq Q$, $X \neq \emptyset$, and terms of $x^{|P|}(\Phi(Q; s, t, x, y, z)-1)$.  The final case is for convex sets $X$ having non-trivial intersection with $P$ and $Q$.  In this case, write $X = U \cap D$.  Then $U = (U \cap P) \cup Q$, $D = (D \cap Q) \cup P$, and $X = (U \cap P) \cup (D \cap Q)$.  Thus $X$ is the union of an arbitrary non-empty upset of $P$ and non-empty downset of $Q$.  Furthermore, every such union is convex in $P \oplus Q$.  Thus it follows, from Table \ref{tab:other_gens_from_poly}, and keeping track of antichain, upset, and downset sizes, that the remaining terms of $\Phi(P \oplus Q ; s, t, x, y, z)$ are exactly the terms of

    \[ x^{|P|} y^{|Q|} (\Phi(P ; s, 1, 1, yw, z/w)|_{w=0}-1)(\Phi(Q ; 1, t, xw, 1, z/w)|_{w=0}-1). \]
    
\end{proof}

Two useful special cases of ordinal sum are the operation of adding a new maximum element above all elements of a given poset $P$, or of adding a new minimum element below all elements of $P$.  Proposition \ref{pp:ordinal_sum} and Example~\ref{eg single elt} give the following.
\begin{cor}
\begin{align*}
    \Phi(P \oplus \bullet ; s, t, x, y, z) & = y(\Phi(P; s, t, x, y, z)-1) + stx^{|P|+1}yz \\
        & \quad + tx^{|P|+1} yz (\Phi(P ; s, 1, 1, yw, z/w)|_{w=0}-1) + 1, \\
    \Phi(\bullet \oplus P ; s, t, x, y, z) & = stxy^{|P|+1}z + x(\Phi(P; s, t, x, y, z)-1) \\
        & \quad + sx y^{|P|+1}z(\Phi(P ; 1, t, xw, 1, z/w)|_{w=0}-1) + 1 .  
\end{align*}
\end{cor}

The Cartesian product $P \times Q$ of posets $P, Q$ is the poset with underlying set $P \times Q$ (the ordinary Cartesian product of sets) and relations $(u_P, u_Q) \leq (v_P, v_Q)$ if and only if $u_P \leq v_P$ and $u_Q \leq v_Q$.  This is a commutative operation, i.e. $P \times Q \cong Q \times P$.  

\begin{rmk}
There is no formula for $\Phi(P \times Q ; s, t, x, y, z)$ in terms of $\Phi(P ; s, t, x, y, z)$ and $\Phi(Q ; s, t, x, y, z)$: consider the posets $P$ and $Q$ of Figure \ref{fig:6_elem_non_iso} with six elements.  We have $\Phi(P ; s, t, x, y, z) = \Phi(Q ; s, t, x, y, z)$, while direct computation gives that $\Phi(P \times [2] ; s, t, x, y, z) \neq \Phi(Q \times [2] ; s, t, x, y, z)$.  The same is true for the simpler polynomials $\Phi(\cdot ; s, t, 1, 1, z)$, $\Phi(\cdot ; 1, 1, x, y, z)$, and $\Phi(\cdot ; 1, 1, 1, 1, z)$, the above example providing the same argument.
\end{rmk}

\subsection{Families of Posets}

Here we discuss formulas for the convex set generating function of specific families of posets.  For the first two we simply write down the polynomial directly from understanding the convex subsets, while the latter two make use of the operations on our generating functions.

\begin{examp}   
    If $P = [n]$ is a chain consisting of $n$ elements, then the convex sets are simply the intervals which we can index by their size and one end, giving

    \begin{align*}
        \Phi([n]; s, t, x, y, z) & = 1 + st \sum_{k=1}^n z^k \sum_{i=0}^{n-k} x^{k+i}y^{n-i} \\
        & = 1 + st\sum_{k=1}^n z^k \frac{x^{n+1}y^k - x^k y^{n+1}}{x-y} \\
        & = 1 + \frac{st}{x-y} \left( x^{n+1} \frac{1-(yz)^{n+1}}{1-yz} - y^{n+1}\frac{1-(xz)^{n+1}}{1-xz} \right).
    \end{align*}
\end{examp}

\begin{examp}
    The convex sets of $P=[n]\times[2]$ can also be worked out directly.  
    
    There are four types of antichains, the empty set, a single element on the lower chain $[n]\times\{1\}$, a single element on the upper chain $[n]\times\{2\}$, or one element on each of the upper and lower chains with the index of the element on the lower chain strictly larger than the index of the element on the upper chain.

    Given an upper antichain for a convex set, for each of the four types observe which types are possible for the lower antichain of the convext set. If the upper antichain is the empty set then so must be the lower antichain for the convex set.  If the upper antichain is of the second type, then so must be the lower antichain and the index of the element of the lower antichain must be weakly less than the index of the index of the element of the upper antichain.  If the upper antichain is of the third type, then the lower antichain can be of the second, third, or fourth type, with all indices weakly less than the index of the element of the upper antichain.  If the upper antichain is of the fourth type,  then the lower antichain is of the second or fourth type.  If it is of the second type then the index of the element in the lower antichain is weakly less than both indices of the upper antichain, while if the lower antichain is of the fourth type then the index of the element of the lower antichain on the lower chain is weakly less than the index of the element of the upper antichain which is on the lower chain and the element of the lower antichain on the upper chain is weakly less than the index of the element of the upper antichain which is on the upper chain.

    Putting this together we get the following formula.
    \begin{align*}
        & \Phi([n]\times[2]; s,t,x,y,z) \\& = 
        1 + st\sum_{1\leq i\leq j\leq n}x^jy^{2(n-i+1)}z^{j-i+1} + st\sum_{1\leq i\leq j\leq n}x^{2j}y^{2(n-i+1)}z^{2(j-i+1)} \\
        & \quad + st\sum_{1\leq i \leq j \leq n}x^{2j}y^{n-i+1}z^{j-i+1} + s^2t\sum_{1\leq i < j \leq k \leq n}x^{2k}y^{2n-i-j+2}  z^{2k-j-i+2} \\
        & \quad + st^2\sum_{1\leq i \leq j < k\leq n} x^{j+k}y^{2(n-i+1)}z^{j+k-2i+2} + s^2t^2\sum_{\substack{1\leq i < j \leq n \\ 1 \leq k < \ell \leq n \\ i \leq k, j \leq \ell}}x^{k+\ell}y^{2n-i-j+2}  z^{k-i+\ell-j+2}.
    \end{align*}
    The number of convex subsets of $[n]\times [2]$ is obtained by setting $s=t=x=y=z=1:$
    \begin{align*}
        \Phi([n]\times[2]; 1,1,1,1,1) & 
        = 1 + 3\binom{n+1}{2} + 2\binom{n+1}{3} + \left(\binom{n}{2} + 3\binom{n}{3} + 2 \binom{n}{4}\right)  \\
        & = 1 + \binom{n+1}{2}\left(2 + \frac{1}{3}\binom{n+2}{2}\right),
    \end{align*}
    where the terms in the final parentheses on the first line are counting pairs as indexed in the final sum of the expression for $\Phi([n]\times[2]; s,t,x,y,z)$.  This agrees after elementary rewriting with the calculation of the number of convex subsets given in Theorem 4.2 of \cite{elder2023toggling}.

    Convex subsets of the product of two chains $P = [n]\times [m]$ have also been studied with the total number given in \cite{barnette2019number}.  There are some nice bijective properties: antichains of $P$ are in correspondence with other combinatorial objects such as alignments of strings, words in a three-letter alphabet, and endpoints of subsequences in walks \cite{bouyssou2024maximal}, while the convex subsets themselves are in bjection with bicoloured Motzkin paths \cite{elizalde2024enumeration}.  Our $\Phi(P ; s, t, x, y, z)$ then enumerates these things with the additional parameters, but itself is apparently too heavy to have nice formulas in this case.
    One of the reasons others have been interested in convex subsets of $[n]\times[m]$ is rowmotion \cite{elder2023toggling, lafreniere2025interval}.
\end{examp}

The power of generating functions is that we can proceed much more easily when can work with our operations as in the following examples.

\begin{examp}\label{eg antichain}
    Write $\bullet^n$ for the antichain of $n$ elements.  Then by Example~\ref{eg single elt} and Proposition~\ref{prop:union} $\Phi(\bullet^n; s, t, x, y, z) = (1+stxyz)^n$.
    By Proposition~\ref{pp:ordinal_sum} 
    \begin{align*}
        \Phi(\bullet^n \oplus \bullet^m; s,t,x,y,z) 
        & = y^m((1+stxyz)^n-1) + x^n((1+stxyz)^m -1 ) \\
        & \quad + x^ny^m(1+syz)^n(1+txz)^m + 1.
    \end{align*}
    This can be generalized as given in the following proposition.
\end{examp} 

\begin{pp}\label{prop complete layers}
\begin{align*}
    & \Phi(\bullet^{n_1}\oplus \bullet^{n_2} \oplus \cdots \oplus \bullet^{n_k}; s,t,x,y,z) \\
    & = 1 + \sum_{i=1}^k x^{n_1+ \cdots + n_{i-1}}y^{n_{i+1} + \cdots + n_k}((1+stxyz)^{n_i}-1) \\
    & \quad + \sum_{1\leq i < j \leq k} x^{n_1+ \cdots + n_{j-1}}y^{n_{i+1}+ \cdots + n_k}z^{n_{i+1}+ \cdots + n_{j-1}}((1+syz)^{n_i}-1)((1+txz)^{n_j}-1).
\end{align*}    
\end{pp}

\begin{proof}
    The proof is by induction on $k$.  The base case is Example~\ref{eg antichain}.  The key that allows nice simplification is how the different terms behave under the substitution of $(s,t,z,y,z)\leftarrow (s,1,1,yw,z/w)$ and evaluation at $w=0$.  Note that $1+stxyz$ becomes $1+syz$ under the substitution and so is unaffected by the evaluation at $w=0$. All the terms of the first sum that have at least one explicit $y$ in the pre-factor will hence have a positive power of $w$ and so be sent to $0$ by the evaluation.  The only term of the first sum that will remain is the last one.  For the terms in the second sum, $1+syz$ is unchanged under the substitution.  Positive powers of $w$ come from the explicit $y$ and negative powers of $w$ come from the explicit $z$ along with the $z$ in the $(1+txz)^{n_j}$ term.  So the only way to obtain $w^0$ is when $j$ is maximal and when from each factor of $(1+txz)^{n_j}$ we take the $txz$.  Thus
    {\allowdisplaybreaks
    \begin{align*}
        & \Phi(\bullet^{n_1}\oplus \bullet^{n_2} \oplus \cdots \oplus \bullet^{n_k}; s,t,x,y,z) \\
        & = \Phi((\bullet^{n_1}\oplus \bullet^{n_2} \oplus \cdots \oplus \bullet^{n_{k-1}} )\oplus \bullet^{n_k}; s,t,x,y,z) \\
        & = y^{n_k}(\Phi(\bullet^{n_1}\oplus \bullet^{n_2} \oplus \cdots \oplus \bullet^{n_{k-1}}; s,t,x,y,z) - 1) + x^{n_1+\cdots + n_{k-1}}((1+stxyz)^{n_k} -1) \\ 
        & \quad + x^{n_1+\cdots + n_{k-1}}y^{n_k}(\Phi(\bullet^{n_1}\oplus \bullet^{n_2} \oplus \cdots \oplus \bullet^{n_{k-1}}; s,1,1,yw,z/w)|_{w=0} - 1)((1+txz)^{n_{k}}-1) \\
        & = \sum_{i=1}^{k-1} x^{n_1+ \cdots + n_{i-1}}y^{n_{i+1} + \cdots + n_{k-1}+n_k}((1+stxyz)^{n_i}-1)  + x^{n_1+\cdots + n_{k-1}}((1+stxyz)^{n_k} -1) \\
        & \quad + \sum_{1\leq i < j \leq k-1} x^{n_1+ \cdots + n_{j-1}}y^{n_{i+1}+ \cdots + n_k}z^{n_{i+1}+ \cdots + n_{j-1}}((1+syz)^{n_i}-1)((1+txz)^{n_j}-1) \\
        & \quad + x^{n_1+\cdots + n_{k-1}}y^{n_k}((1+txz)^{n_{k}}-1) \\
        & \qquad \qquad \cdot \left( 1 + ((1+syz)^{n_{k-1}}-1) + \sum_{1\leq i < k-1} z^{n_{i+1}+ \cdots + n_{k-1}}((1+syz)^{n_i}-1) - 1\right)\\
        & = 1 + \sum_{i=1}^k x^{n_1+ \cdots + n_{i-1}}y^{n_{i+1} + \cdots + n_k}((1+stxyz)^{n_i}-1) \\
    & \quad + \sum_{1\leq i < j \leq k} x^{n_1+ \cdots + n_{j-1}}y^{n_{i+1}+ \cdots + n_k}z^{n_{i+1}+ \cdots + n_{j-1}}((1+syz)^{n_i}-1)((1+txz)^{n_j}-1).
    \end{align*}
    }
\end{proof}

Evaluating at $s=t=x=y=z=1$ we obtain the number of convex sets of ordinal sums of antichains
\[
  \Phi(\bullet^{n_1}\oplus \bullet^{n_2} \oplus \cdots \oplus \bullet^{n_k}; 1,1,1,1,1) = 1 + \sum_{i=1}^k(2^{n_i}-1) + \sum_{1 \leq i < j \leq k} (2^{n_i}-1)(2^{n_j}-1),
\]
agreeing with the calculation given in Theorem 3.3 of \cite{elder2023toggling}.

\section{Posets Generated by $+$ and $\oplus$}

In this section we focus on families of posets that are generated by the disjoint union and ordinal sum operations.  Let $F$ be a set of non-empty finite posets such that:

\begin{enumerate}
\item each member of $F$ has a distinct convex set generating polynomial, 
\item the polynomials of members of $F$ do not factor into non-constant polynomials with constant coefficient $1$, and
\item no member of $F$ can be written as the union or ordinal sum of two or more non-empty posets.
\end{enumerate}

We then let $\mathcal{F}$ be the set of posets constructed by iteratively applying $+$ and $\oplus$ to members of $F$.  A well-known example of such an $\mathcal{F}$ is the class of \emph{Series-Parallel (SP)} posets.  This class of posets corresponds to $F = \{\bullet\}$, the singleton set containing the poset consisting of one element.  Equivalently, SP posets are exactly those which are $N$-free, that is, not containing a subposet isomorphic to the $N$ poset as shown in Figure \ref{fig:Nposet}.  SP posets are a well-studied family of posets: a generating series for the number of these posets was found in \cite{stanley1974enumeration}, and the asymptotics were computed in \cite{bayoumi1989asymptotic}.  Further below, we mention some results related to distinguishing SP posets.

\begin{figure}[h]
    \centering
    \input{figs/Nposet}
    \caption{The $N$ poset.}
    \label{fig:Nposet}
\end{figure}
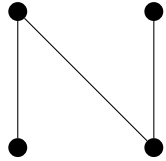

As a consequence of condition 3, each element of $\mathcal{F}$ is uniquely described by its \emph{composition tree} (see, for example, \cite{schrader1997setup}), which is a rooted tree describing the structure of $+$ and $\oplus$ operations used to construct the poset.  
We give an example of a composition tree for a SP poset in Example \ref{ex:comp_tree}.

\begin{examp}
    Consider the SP poset $P$ shown in the left of Figure \ref{fig:ex_of_sp_tree}.  Since $P$ is a disjoint union of two posets that are themselves not a disjoint union, the root of the composition tree $T$ in the right of Figure \ref{fig:ex_of_sp_tree} is a vertex labelled $+$ with two neighbours.  The two-element component of $P$ is the ordinal sum of the single-element posets with vertices $a$ and $b$, so the neighbour of the root corresponding to this component is a vertex labelled $\oplus$ which has two leaves, labelled $a$ and $b$.  Similarly, the other component of $P$ is the ordinal sum of the antichain $\{c, d\}$ with $\{e\}$.  Note that the plane structure of the tree at each vertex corresponding to a $\oplus$ operation matters, since $\{c, d\} \oplus \{e\} \neq \{e\} \oplus \{c, d\}$.

    \begin{figure}[h]
        \centering
        \input{figs/SP_comp_tree}
        \caption{An SP poset and its composition tree.}
        \label{fig:ex_of_sp_tree}
    \end{figure}
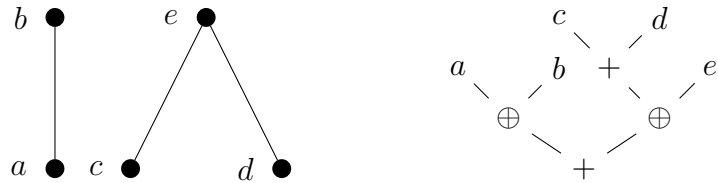
    \label{ex:comp_tree}
\end{examp}

While our convex set generating polynomial does not distinguish all posets, in this section we prove that it does distinguish all posets in any family $\mathcal{F}$ satisfying conditions 1-3 above.  The question of distinguishing the family of SP posets has been investigated previously: in \cite{gordon1996series} it was shown that a strict subclass of SP posets are distinguishable via their \emph{Tutte polynomial}, while it was left undetermined whether the Tutte polynomial distinguishes all SP posets.  Recently, the result of \cite{gordon1996series} was extended: in \cite{luo2026tutte} a larger subclass $H$ of SP posets than the subclass in \cite{gordon1996series} was defined.  It was shown that if $P \in H$ and $Q$ is SP, then equality of the Tutte polynomials of $P$ and $Q$ implies $P \cong Q$.

Our proof that the convex set generating polynomial distinguishes posets in $\mathcal{F}$ will go as follows: first, we show that we can recover, from the polynomial of an ordinal sum of posets, the polynomial of each summand.  We then prove that the polynomial of an ordinal sum does not non-trivially factor.  So, given just the polynomial we can determine if a poset is a disjoint union or an ordinal sum and the polynomials of each factor of these operations.  Inductively, in this way we construct the composition tree of the poset, hence uniquely determining the poset.  The following lemma, which we state without proof, will be useful.

\begin{lm}
    (Lemma 5 of \cite{gordon1996series}) A poset $R$ is an ordinal sum of two non-empty posets if and only if there is exactly one downset of size $k$ in $R$ for some $0< k < |R|$.  For such a poset $R$, $R = I \oplus (R \setminus I)$ where $I$ is the unique downset of size $k$ in $R$.
    \label{lm:gordon_ideal}
\end{lm}

\begin{lm}
    Let $P = P_1 \oplus \cdots \oplus P_r$ be an ordinal sum of posets where each $P_i$ cannot be written as an ordinal sum of non-empty posets.  Then $\Phi(P ; s, t, x, y, z)$ determines $\Phi(P_i ; s, t, x, y, z)$ for all $i$.
    \label{lm:downset_gets_ordinals}
\end{lm}

\begin{proof}
    Consider the specialization to downsets $\mathcal{D}(P ; x) = \Phi(P ; 1, 1, xw, 1, 1/w)|_{w=0}$.  Let $b_1 < \cdots < b_\ell < |P|$ be the exponents of the terms in $\mathcal{D}(P ; x)$ with coefficient $1$ (note that $|P|$ is obtained as the largest exponent of $x$).  By Lemma \ref{lm:gordon_ideal}, $\ell = r$ and $|P_i| = b_i - b_{i-1}$ for all $i$ ($b_0 = 0$).  Thus the non-constant terms of $\Phi(P_i ; s, t, x, y, z)$ correspond exactly to the terms of $\Phi(P ; s, t, x, y, z)$ that have $x$ exponent strictly larger than $b_{i-1}$ and at most $b_i$.  To be more precise, let $\alpha s^\beta t^\gamma x^b y^a z^c$ be a non-constant term of $\Phi(P ; s, t, x, y, z)$ with $b_{i-1} < b \leq b_i$.  Then $\alpha s^\beta t^\gamma x^{b - b_{i-1}} y^{a - |P| + b_{i}} z^c$ is a non-constant term of $\Phi(P_i ; s, t, x, y, z)$ and all non-constant terms of $\Phi(P_i ; s, t, x, y, z)$ appear in this way. 
\end{proof}

\noindent The next lemma, along with condition 2, establishes that the convex set generating polynomial of $P \in \mathcal{F}$ factors non-trivially if and only if $P$ is a disjoint union.

\begin{lm}
    Let $P = P_1 \oplus P_2$ for any non-empty posets $P_1, P_2 \in F$.  Then $\Phi(P ; s, t, x, y, z)$ does not factor into two non-constant polynomials each with constant coefficient 1.
    \label{lm:ord_sum_factor}
\end{lm}

\begin{proof}
    Toward contradiction, suppose $\Phi(P ; s, t, x, y, z) = GH$ for non-constant polynomials $G, H$ each with constant term one.  Recall the Newton polytope of $\Phi(P ; s, t, x, y, z)$:
    
    \[ \mathcal{N}_P = {\rm conv}\{(a_1, \ldots, a_5) \;:\; [s^{a_1} t^{a_2} x^{a_3} y^{a_4} z^{a_5}]\Phi(P ; s, t, x, y, z) \neq 0\}. \]

    \noindent Since $\Phi(P ; s, t, x, y, z)$ factors, $\mathcal{N}_P$ factors as the Minkowski sum of polytopes

    \[ \mathcal{N}_P = \mathcal{N}_G \oplus \mathcal{N}_H = \{ u + v \;:\; u \in \mathcal{N}_G, v \in \mathcal{N}_H \} \]

    \noindent where $\mathcal{N}_G$ is the Newton polytope of $G$ and similarly for $H$ \cite{ostrowski1975multiplication}.  Thus $(0, \ldots, 0) \in \mathcal{N}_G, \mathcal{N}_H$.  Since every non-constant term of $\Phi(P ; s, t, x, y, z)$ has an exponent of at least one in each variable, every point in $\mathcal{N}_P$ and thus each of $\mathcal{N}_G, \mathcal{N}_H$ except for the origin has positive value in each coordinate.  If, without loss of generality, there was a point $u \in \mathcal{N}_G$ with both a non-zero and zero coordinate, then $u = u + 0 \in \mathcal{N}_P$, a contradiction.  Therefore each non-constant term of $G$ and $H$ has an exponent of at least one in each variable.

    Consider a term $\alpha s^a t x^{|P_1|+c} y^{|P|} z^{|P_1|+c}$ of $\Phi(P ; s, t, x, y, z)$ corresponding to the convex set $X = V(A) \cap \Lambda(w)$ where $A$ is the antichain of minimal elements of $P_1$ and $w \in P_2$.  Thus $|A| = a$ and $c > 0$.  Since $t$ has an exponent of one, by the above observation either this term is a term of $G$ or a term of $H$.  Without loss of generality, say it is a term of $G$.  Then $(a, 1, |P_1|+c, |P|, |P_1|+c) \in \mathcal{N}_G$.  But taking any non-origin point $v \in \mathcal{N}_H$, the fourth coordinate of $v$ is positive and hence $(a, 1, |P_1|+c, |P|, |P_1|+c) + v$ has a fourth coordinate strictly greater than $|P|$.  This is a contradiction, since any point in $\mathcal{N}_P$ has fourth coordinate at most $|P|$.
\end{proof}

Now we have enough information to inductively construct the composition tree of a poset in $\mathcal{F}$ from its polynomial.

\begin{tm}
    Let $\mathcal{F}$ be a family of posets where $F$ satisfies conditions 1-3.  If $P, Q \in \mathcal{F}$ and $\Phi(P; s, t, x, y, z) = \Phi(Q; s, t, x, y, z)$, then $P \cong Q$.
    \label{tm:F_dist}
\end{tm}

\begin{proof}
    First we show how to construct $P$ from $\Phi(P ; s, t, x, y, z)$.  The proof is by induction on the depth of the decomposition tree of $P$ as an element of $\mathcal{F}$.  If the decomposition tree of $P$ is a single leaf, then $P \in F$ and the result holds by condition 1.  So, suppose $P \not \in F$ and that the result holds for posets with shallower decomposition trees than $P$.  Let $\Phi(P ; s, t, x, y, z) = F_1 F_2 \cdots F_r$ be the unique minimal factorization of $\Phi(P ; s, t, x, y, z)$ scaled so all factors have constant coefficient 1.  From Lemma \ref{lm:ord_sum_factor}, we know that the polynomials of ordinal sums do not factor; it follows from this and condition 2 that $P$ has exactly $r$ components $P_1, \ldots, P_r$, i.e.\ $P = P_1 + \cdots + P_r$. From Proposition \ref{prop:union}, we have $F_i = \Phi(P_i ; s, t, x, y, z)$ (after some labelling of the $P_1, \ldots, P_r$).  For each $i$, since $P_i$ is not a disjoint union of smaller, non-empty posets, it is an ordinal sum.  So $P_i = Q_{i, 1} \oplus \cdots \oplus Q_{i, \ell}$, where each $Q_{i, j}$, $1 \leq j \leq \ell$, is not an ordinal sum of smaller, non-empty posets.  By Lemma \ref{lm:downset_gets_ordinals}, we can recover $\Phi(Q_{i, j} ; s, t, x, y, z)$ for each $j$ from $\Phi(P_i ; s, t, x, y, z)$.  By induction, we can uniquely determine each $Q_{i, j}$, and hence $P$.

    It immediately follows that if $\Phi(P; s, t, x, y, z) = \Phi(Q; s, t, x, y, z)$, then $P \cong Q$.
\end{proof}

In particular, we can distinguish SP posets from their convex set generating polynomials.

\begin{cor}
    Let $P$ and $Q$ be SP posets.  If $\Phi(P; s, t, x, y, z) = \Phi(Q; s, t, x, y, z)$, then $P \cong Q$.
\end{cor}

While $\Phi(\cdot ; s, t, x, y, z)$ distinguishes all posets in $\mathcal{F}$, it remains open whether any of its specializations will do the same, even for SP posets.  To probe exactly how much information contained in the convex set generating polynomial is needed to distinguish SP posets, we investigated the SP-distinguishing power of some of the specializations of $\Phi(\cdot ; s, t, x, y, z)$ given in Table \ref{tab:other_gens_from_poly}.  Computations reveal that the polynomial of the first pf-invariant $[z]\Phi(\cdot ; 1, 1, x, y, z)$ does not distinguish a pair of SP posets with seven elements (nor the pair of dual posets), see Figure \ref{fig:sp_opt4_nondis}.  The ordinary interval polynomial $[st]\Phi(\cdot ; s, t, x, y, z)$ does not distinguish a pair of SP posets with ten elements, see Figure \ref{fig:sp_pos_ord_poly_nondis}.  The downset polynomial $\Phi(\cdot ; 1, 1, xw, 1, 1/w)|_{w=0}$ (and dually the upset polynomial) does not distinguish the pair of SP posets with five elements shown in Figure \ref{fig:downset_counter}.  The specializations $\Phi(\cdot ; 1, 1, x, y, z)$, $\Phi(\cdot ; s, t, 1, 1, z)$, $\Phi(\cdot ; 1, t, xw, 1, 1/w)|_{w=0}$ (dually $\Phi(\cdot ; s, 1, 1, yw, 1/w)|_{w=0}$), and $\Phi(\cdot ; 1/w, 1, x, y, w)|_{w=0}$ (dually $\Phi(\cdot ; 1, 1/w, x, y, w)|_{w=0}$) distinguish all SP posets with ten elements or fewer, but it remains inconclusive whether these specialization continue to distinguish SP posets of larger cardinality.  We remark that Lemma \ref{lm:downset_gets_ordinals} holds for any specialization of $\Phi(\cdot ; s, t, x, y, z)$ that specializes further to the downset polynomial $\mathcal{D}(\cdot ; x)$.  Thus Lemma \ref{lm:downset_gets_ordinals} holds for all of the above specializations except potentially $\Phi(\cdot; s, t, 1, 1, z)$.

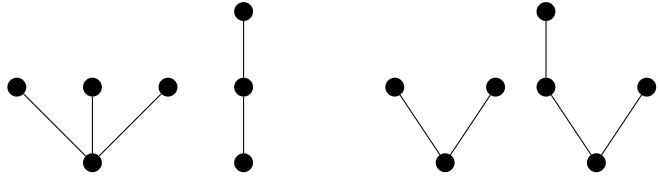
\begin{figure}[ht]
    \centering
    \input{figs/pf_not_dist}
    \caption{Two SP posets that are not distinguished by the polynomial form of the first pf-invariant $[z]\Phi(\cdot ; 1, 1, x, y, z)$.  The pair of dual posets is also not distinguished by this specialization.}
    \label{fig:sp_opt4_nondis}
\end{figure}

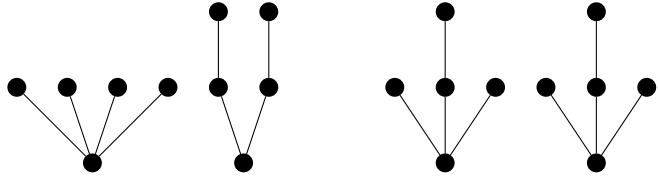
\begin{figure}[ht]
    \centering
    \input{figs/int_not_dist}
    \caption{Two SP posets that are not distinguished by the ordinary interval polynomial $[st]\Phi(\cdot ; s, t, x, y, z)$.}
    \label{fig:sp_pos_ord_poly_nondis}
\end{figure}

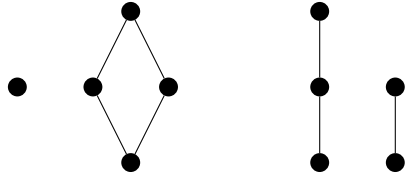
\begin{figure}[ht]
    \centering
    \input{figs/downset_counter}
    \caption{Two SP posets that are not distinguished by the downset polynomial $\Phi(\cdot ; 1, 1, xw, 1, 1/w)|_{w=0}$.}
    \label{fig:downset_counter}
\end{figure}

\section{Applications to Causal Sets}

We now turn to how convex sets of posets find use in the causal set approach to quantum gravity.  This was the motivating problem that led us to define our convex set generating polynomial.  We will give a brief introduction to Causal Set Theory, discuss the previous results which motivated the present work, and discuss the utility  of convex sets in Causal Set Theory.  For a more complete and in-depth account of Causal Set Theory and its development, see, for example, \cite{surya2019causal}.

\subsection{Causal Set Theory}

Causal Set Theory (CST) is a theory of quantum gravity in which spacetime, which would more usually be a Lorentzian manifold, is taken to be a locally-finite poset, or a \emph{causal set}.  The smooth manifolds of General Relativity, therefore, are continuous approximations of this discrete structure.  The causal order of the spacetime, which is determined by the metric, is represented by the partial order relation of the causal set, and the volume of a region in the spacetime is represented by the cardinality of a subset of the causal set.  The condition of local-finiteness ensures that regions of spacetime that have finite volume correspond to finite cardinality subsets of the causal set.

The central claim of CST is that the causal order, along with a notion of volume, captures all the information about a spacetime.  An open conjecture, called the \emph{Hauptvermutung}, claims that a continuous spacetime that approximates a causal set is unique up to a notion of approximate isometry.  This conjecture remains open largely due to the difficulty of rigorously defining this notion of approximate isometry.

Not all causal sets resemble an actual spacetime manifold.  Famously, the set of finite causal sets (finite posets) is asymptotically dominated by those with three layers \cite{kleitman1975asymptotic} as the causal set cardinality tends to infinity.  Hence a causal set taken at random will not resemble a spacetime manifold, i.e. will not be \emph{manifold-like}.  Thus CST researchers have developed alternative techniques for constructing manifold-like causal sets.  One such technique which will be relevant for what follows is called \emph{sprinkling}.  Here, points are chosen from an underlying spacetime manifold according to a random Poisson point process.  This process ensures that, on average, the resulting causal set is invariant under Lorentz transformations (which is not true of regular grids embedded into a spacetime).  While sprinkling produces manifold-like causal sets, one drawback is that it assumes a pre-existing spacetime.  Thus if one wishes to study the dynamics of causal sets, a different approach is needed.

A basic problem and pillar of CST research is the determination of physical properties of a causal set.  That is, extracting the analogues of physical characteristics of a spacetime just from the partial order of the causal set.  We have already seen above that the volume of a spacetime corresponds to subset cardinality, but there are plenty more properties to be desired, such as: geodesics between spacetime events, Minkowski dimension, and Ricci curvature.  Determining Minkowski dimension is of particular interest to us, as discussed below.  In pure poset language the Minkowsi dimension corresponds to dimension as a sphere order \cite{trotter2002combinatorics}.

\subsection{Motivating Work}
\label{sec:cst_motivating}

Given a causal set, its \emph{Minkowski dimension} is the smallest dimension of Minkowski space that it can be embedded into such that all relations and non-relations are preserved.  This notion of dimension differs from the more combinatorial \emph{order dimension}, except in the two-dimensional case when these definitions agree \cite{meyer1988dimension}.  In the present literature there exist numerous methods for estimating this dimension (see, for example, \cite{surya2019causal}, \cite{aghili2019discrete}).  Here, we will focus on one method involving the distribution of interval sizes.  This work was the original motivation for our investigations into convex sets.

In \cite{glaser2013towards}, the authors determined the asymptotic behaviour of the distribution of interval sizes of a causal set $C$ sprinkled into a causal diamond of $d$-dimensional Minkowski space.  Let $N_k^d$ be the number of intervals of cardinality $k+2$ in such an $n$-element sprinkling.  The authors determined that for $k \geq 0$,

\[ s_k^d = \lim_{n \to \infty} \frac{N_k^d}{N_0^d} = \frac{\Gamma(\frac{2}{d}+k)}{\Gamma(\frac{2}{k}) \Gamma(k+1)} = \binom{\frac{2}{d}+k-1}{k}. \]

\noindent Thus if we take the generating series of these quantities, we obtain

\[ \quad S_d(z) \equiv \sum_{k \geq 0} s_k^d z^k = \frac{1}{(1-z)^{\frac{2}{d}}} \]

\noindent which packs the Minkowski dimension into an elegant closed expression.  Note that $[z]S_d(z) = s_1^d = \frac{2}{d}$, which is the limit of the proportion of $3$-element intervals divided by the number of $2$-element intervals.

This method, however, does have drawbacks.  First, the convergence of the values $s_k^d$ is slow \cite{aghili2019discrete}.  Second, it cannot capture some additional information such as temporal symmetry, or lack thereof, for a causal set and its dual (time-reversal) have the same distribution of interval sizes.

The same interval size distributions were used in \cite{surya2026closeness} to define a \emph{closeness} function on causal sets.  In short, the closeness between causal sets $C$ and $C'$ (these causal sets need not be constructed from sprinkling, so we drop the superscript $d$) was defined to be

\[  \mathcal{D}_n(C, C') = \sum_k |N_k - N_k'| \]

\noindent (alternatively, this $L^1$-norm could be replaced with any $L^r$-norm).  Empirical evidence, \cite{surya2026closeness}, showed that this closeness function detects when causal sets have the same Minkowski dimension and global topology.  However, it suffers from the same issue of temporal symmetry as above.

\subsection{Convex Sets in Causal Sets}

The results of the methods discussed in Section \ref{sec:cst_motivating} led us to wonder how convex sets, being generalizations of intervals, could extend these results, and possibly overcome some of the issues they faced.  Let us first note alternative views of convex sets that may suggest reasons to be interested in convex sets from a physical point of view.  A simple translation of the definition shows that convex sets are those sets which which contain every possible path between any pair of elements in the set.  In a continuous spacetime, this means a set for which any time-like curve between a pair of points in the set is also contained in the set.  An equivalent definition of a convex set is: a set $X \subseteq P$ is convex if for any $u \in P$, $u \not \in X$, at most one of $X \cap V(u)$ and $\Lambda(u)$ is non-empty.  Thus to some outside observer, a convex set cannot simultaneously lie to their future and past.

To our knowledge, a general treatment of convex sets of causal sets does not appear in the literature.  Some consideration of convex sets appeared in a variation of \emph{covtree}, which captures the dynamics of classical causal set growth models \cite{zalel2024covariant}.  A special kind of convex set, called a \emph{thickened antichain}, was shown to be associated to covariant observables of the same growth models in \cite{major2006spatial}.  A special case of convex sets, called \emph{causally closed} sets, appears in \cite{casini2002logic}.  Here, an orthomodular lattice of causally closed sets of a spacetime is constructed and analyzed.  Furthermore, convex sets of spacetime appear in the axioms of Algebraic Quantum Field Theory (AQFT), being those sets for which algebras of observables are attached \cite{fewster2020algebraic}.

The geometry of special kinds of convex sets has been previously studied.  There is a well-known connection between the volume of causal diamonds and their geometry.  Expanding the volume of a causal diamond in a curved spacetime with respect to the proper time between the endpoints of the diamond gives a series whose coefficients involve the curvature at the centre of the diamond \cite{myrheim1978statistical, gibbons2007geometry}.  A causal set analogue of this derivation was done in \cite{roy2013discrete}.  Using the expected abundance of chains of each length in a causal set sprinkled into a causal diamond of a curved spacetime, the authors found expressions for the curvature terms.  Generalizing the previous work to causal cones, the author of \cite{jubb2017geometry} expanded the volume of a causal cone (a region bounded by the lightcone originating from a point and a spatial hypersurface) in a curved spacetime and found expressions for intrinsic and extrinsic curvature terms.  Applying these expressions to sprinkled causal sets gave expressions relating geometric quantities to the expected number of elements with a given cardinality of its past (similarly future).

In the remainder of this section we describe some preliminary work toward extending the results of Section \ref{sec:cst_motivating} to convex sets.  Given a (finite) causal set $C$, let $X_{i, j, k}$ be the number of convex sets $X$ of $C$ with $|\min(X)| = i, |\max(X)| = j$, and $|X| = k$.  That is, $X_{i, j, x} = [x^j y^i z^k]\Phi(C ; s, t, 1, 1, z)$.  First, we deal with the problem of temporal symmetry.  For causal sets sprinkled into spacetime regions that are not temporally symmetric, we expect the off-diagonal distributions of $X_{i, j, k}$, where $i \neq j$, to detect this lack of symmetry.  This expectation is supported by results from small causal sets, shown in Figure \ref{fig:expconD=2}.  Here we have sprinkled $50$ points into a region of $2$-dimensional Minkowski space and computed the distributions of convex set sizes for $(i, j) = (1, 1), (1, 2), (2, 1)$.  The distributions are averaged over $50$ simulations.  The region of sprinkling for Figure \ref{fig:expD=2} was the lower half of a causal diamond centred at the origin, and the region for Figure \ref{fig:conD=2} was the upper half of the same diamond.  We see a clear difference in the distributions for $(i, j) = (1, 2)$ and $(i, j) = (2, 1)$ in each region, and furthermore they agree with our expectation that the causal sets sprinkled into the lower region should have more $(i, j) = (1, 2)$ convex sets than $(i, j) = (2, 1)$ convex sets and vice versa.  We also clearly see that intervals, the convex sets where $(i, j) = (1, 1)$, do not distinguish the sprinklings in the different regions.

\begin{figure}[h]
    \centering
    \begin{subfigure}{0.48\linewidth}
        \centering
        \includegraphics[scale=0.5]{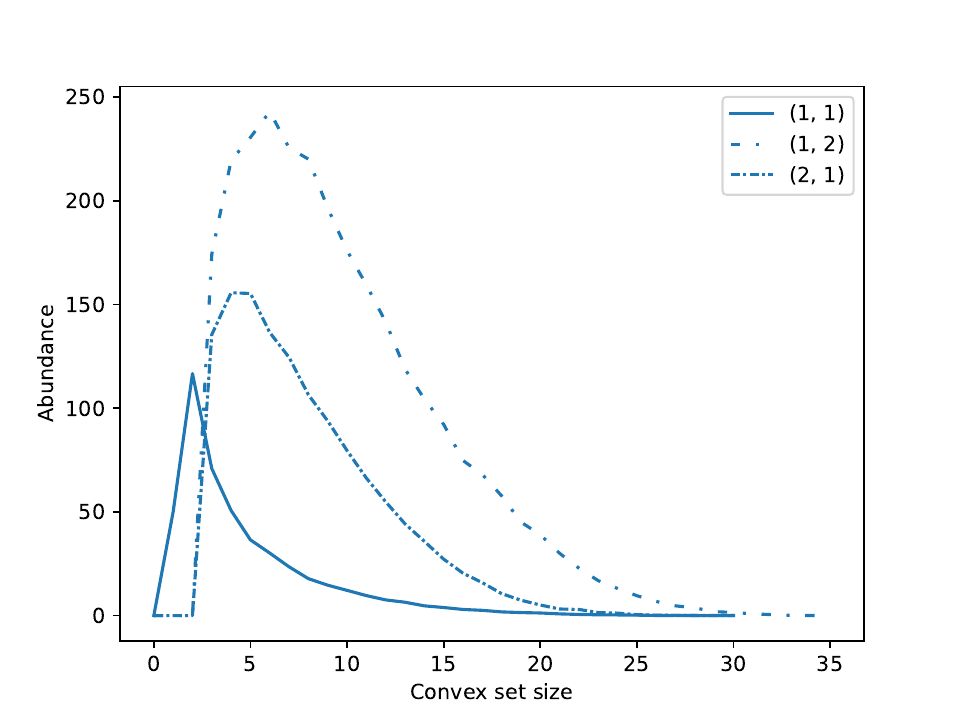}
        \caption{}
        \label{fig:expD=2}
    \end{subfigure}
    \begin{subfigure}{0.48\linewidth}
        \centering
        \includegraphics[scale=0.5]{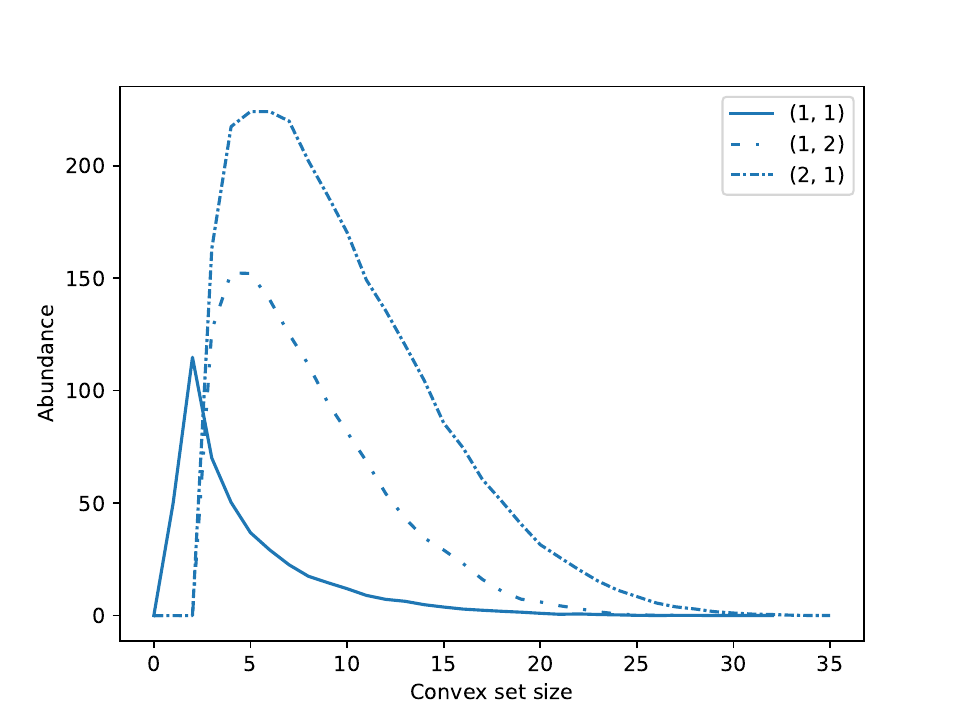}
        \caption{}
        \label{fig:conD=2}
    \end{subfigure}
    \caption{Convex set distributions for causal sets sprinkled into temporally-asymmetric regions of $2$-dimensional Minkowski spacetime.  $X_{i, j, k}$ is plotted with $k$ along the horizontal axes for fixed pairs $(i, j)$ as shown in the legends.  In (a) the region of sprinkling is the lower half of a causal diamond centred on the origin, while in (b) the region is the upper half of the diamond.  In both plots, sprinklings had $50$ elements and distributions were averaged over $50$ sprinklings.}
    \label{fig:expconD=2}
\end{figure}

Our polynomial can also be used to define a closeness function on causal sets that extends what was given in \cite{surya2026closeness}.  Given two causal sets $C$ and $C'$, we can define

\[ \mathcal{D}_{convex}(C, C') = |\Phi(C ; s, t, x, y, z) - \Phi(C' ; s, t, x, y, z)| \]

\noindent where by the absolute value of a polynomial, we mean taking the absolute value of its coefficients.  We can then evaluate $\mathcal{D}_{convex}(C, C')$ at certain values of the parameters to obtain numerical values.  Thus far we do not have results involving sprinkled causal sets.  However, we have observed an interesting phenomenon about $\mathcal{D}_{convex}$ for small posets.  In Figure \ref{fig:cls} we plot histograms showing the closeness values between posets with seven elements.  More specifically, given posets $P, P'$ with $|P| = |P'| = 7$, we compute the value $\mathcal{D}_{convex}(P, P')|_{s=t=x=y=z=1}$.  The top histogram is the histogram of these values for all such pairs of $P$ and $P'$.  In the bottom plot we only consider terms of $\Phi(\cdot ; s, t, x, y, z)$ corresponding to intervals, i.e. the polynomial $[st]\Phi(\cdot ; s, t, x, y, z)$.  Thus the closeness function in this case is

\[ \mathcal{D}_{interval}(P, P') = |[st]\Phi(P ; s, t, x, y, z) - [st]\Phi(P' ; s, t, x, y, z)| \]

\noindent and we further evaluate $\mathcal{D}_{interval}(P, P')|_{x=y=z=1}$.  What we observe is a comb-like pattern to the histogram for the $\mathcal{D}_{convex}$ plot, where there are seemingly two distributions of differing means and amplitudes that interlace each other.  This behaviour is not observed for $\mathcal{D}_{interval}$.  As of yet we do not have an explanation for this phenomenon, so it remains a combinatorial oddity.

\begin{figure}[h]
    \centering
    \includegraphics[scale=0.6]{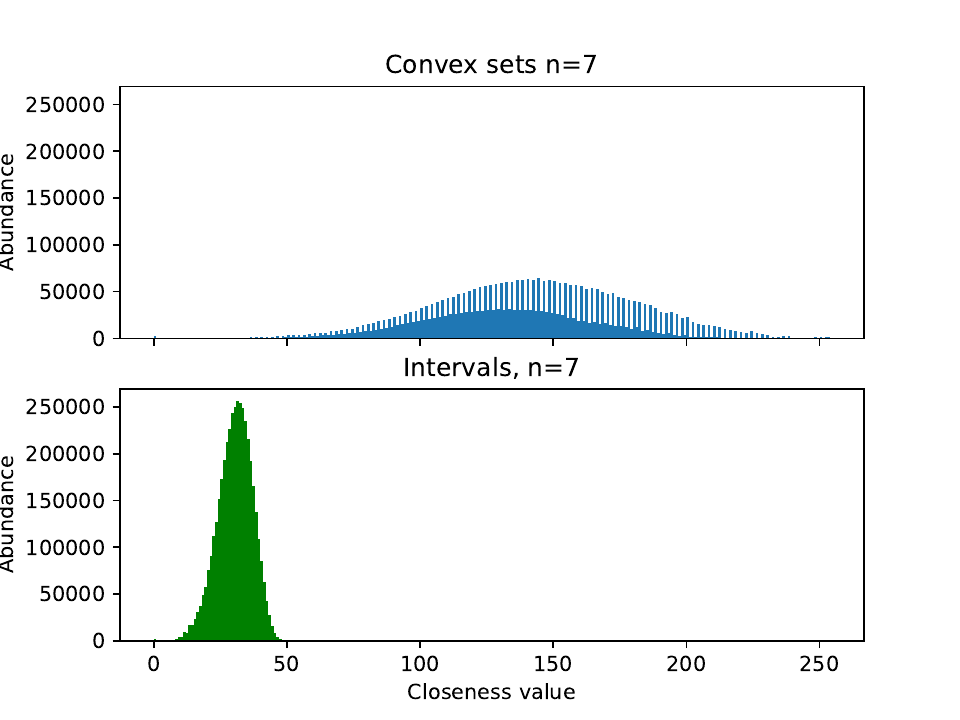}
    \caption{Histograms showing the occurrences of closeness values between posets with seven elements.  The top uses all terms of $\Phi(\cdot ; s, t, x, y, z)$, and sets $s=t=x=y=z=1$ after absolute values have been applied.  The bottom uses the terms corresponding to intervals, i.e. $[st]\Phi(\cdot ; s, t, x, y, z)$, and similarly sets $x=y=z=1$ after absolute values are applied.}
    \label{fig:cls}
\end{figure}

While convex sets seem to hold important physical information of a causal set, we must note a difficulty in gathering numerical evidence to support this.  Computing the polynomial $\Phi(P ; s, t, x, y, z)$ is computationally time-consuming due to how many convex sets $P$ has.  Compare this to only computing the terms corresponding to intervals, of which there are at most $\sim n^2/2$ for a poset with $n$ elements.  Thus it is not realistic to compute the full polynomial for large posets.  However, to have confidence in numerical results involving sprinkled causal sets, the causal sets are ideally large ($\geq 1000$ elements).  Despite this difficulty, we offer two possible ways forward: $1)$ restrict to very small or very large convex sets (relative to the size of entire causal set), of which there are few, or $2)$ sample a small subset of convex sets.  It is possible that in option $1)$, the small and large convex sets will struggle to capture information that those with intermediate size may.

We will be pursuing these and related ideas in upcoming work with Sumati Surya.


\section*{AI Declaration}

The only AI used in the preparation of this paper was Google's AI Overview, which pointed the authors to the notion of the Newton polytope.  All observations, proofs, code-writing, exposition, and literature searches were performed by the authors.

\section*{Acknowledgements}

The authors thank Sumati Surya for useful conversations, Kevin Purbhoo for suggesting there may not be a cartesian product formula for the convex set generating polynomial, and Adrian Segovia for pointing us to the terminology of a convex set.  The authors are grateful to Rafael Sorkin and the discussions we were able to have had with him regarding the pf-invariant(s).

\bibliographystyle{plain}
\bibliography{cst_refs}

\appendix

\newpage
\section*{Appendix A: Undistinguished Posets with Seven Elements}
\label{app:A}

Here we list the 16 pairs of posets with seven elements that are not distinguished by the convex set generating polynomial.

\begin{table}[h]
    \centering
    \begin{tabular}{|c||c|}
        \hline
        \input{figs/p1} & \input{figs/p9} \\
        \hline
        \input{figs/p2} & \input{figs/p10} \\
        \hline
        \input{figs/p3} & \input{figs/p11} \\
        \hline
        \input{figs/p4} & \input{figs/p12} \\
        \hline
    \end{tabular}
    \caption{Eight pairs of posets not distinguished by the polynomial.}
    \label{tab:app_16pairs_pt1}
\end{table}

\begin{table}[h]
    \centering
    \begin{tabular}{|c||c|}
        \hline
        \input{figs/p5} & \input{figs/p13} \\
        \hline
        \input{figs/p6} & \input{figs/p14} \\
        \hline
        \input{figs/p7} & \input{figs/p15} \\
        \hline
        \input{figs/p8} & \input{figs/p16} \\
        \hline
    \end{tabular}
    \caption{An additional eight pairs of posets not distinguished by the polynomial.}
    \label{tab:app_16pairs_pt2}
\end{table}

\end{document}

%% file: figs/poly_example.tex
\begin{tikzpicture}[scale=1.2]
\tikzset{vertex/.style={circle, inner sep=0pt, outer sep=0pt, minimum width=0.25cm,fill=black},
rvertex/.style = {vertex, fill=red},
bvertex/.style = {vertex, fill=blue},
gvertex/.style = {vertex, fill=green},
nedge/.style = {-}}

\node[vertex, label={[label distance=0.1cm]270:$a$}] (A) at (0, 0) {};
\node[vertex, label={[label distance=0.1cm]180:$b$}] (B) at (-0.5, 1) {};
\node[vertex, label={[label distance=0.1cm]0:$c$}] (C) at (0.5, 1) {};
\node[vertex, label={[label distance=0.1cm]90:$d$}] (D) at (0, 2) {};

\path [-] (A) edge node {} (B);
\path [-] (A) edge node {} (C);
\path [-] (B) edge node {} (D);
\path [-] (C) edge node {} (D);


\end{tikzpicture}

%% file: figs/6elem_not_dist.tex
\begin{tikzpicture}[scale=0.9]
\tikzset{vertex/.style={circle, inner sep=0pt, outer sep=0pt, minimum width=0.25cm,fill=black},
rvertex/.style = {vertex, fill=red},
bvertex/.style = {vertex, fill=blue},
gvertex/.style = {vertex, fill=green},
nedge/.style = {-}}

\node[vertex] (a) at (0, 0) {};
\node[vertex] (b) at (2, 0) {};
\node[vertex] (c) at (0, 2) {};
\node[vertex] (d) at (2, 2) {};
\node[vertex] (e) at (0, 4) {};
\node[vertex] (f) at (2, 4) {};

\path [-] (a) edge node {} (c);
\path [-] (c) edge node {} (b);
\path [-] (b) edge node {} (d);
\path [-] (c) edge node {} (e);
\path [-] (e) edge node {} (d);
\path [-] (d) edge node {} (f);


\node[vertex] (1) at (4, 0) {};
\node[vertex] (2) at (6, 0) {};
\node[vertex] (3) at (4, 2) {};
\node[vertex] (4) at (6, 2) {};
\node[vertex] (5) at (5, 3) {};
\node[vertex] (6) at (5, 4) {};

\path [-] (1) edge node {} (3);
\path [-] (3) edge node {} (2);
\path [-] (2) edge node {} (4);
\path [-] (4) edge node {} (5);
\path [-] (5) edge node {} (6);
\path [-] (1) edge node {} (6);


\end{tikzpicture}

%% file: figs/Nposet.tex
\begin{tikzpicture}[scale=0.9]
\tikzset{vertex/.style={circle, inner sep=0pt, outer sep=0pt, minimum width=0.25cm,fill=black},
rvertex/.style = {vertex, fill=red},
bvertex/.style = {vertex, fill=blue},
gvertex/.style = {vertex, fill=green},
nedge/.style = {-}}

\node[vertex] (a) at (0, 0) {};
\node[vertex] (b) at (2, 0) {};
\node[vertex] (c) at (0, 2) {};
\node[vertex] (d) at (2, 2) {};

\path [-] (a) edge node {} (c);
\path [-] (c) edge node {} (b);
\path [-] (b) edge node {} (d);


\end{tikzpicture}

%% file: figs/SP_comp_tree.tex
\begin{tikzpicture}[scale=1.0]
\tikzset{vertex/.style={circle, inner sep=0pt, outer sep=0pt, minimum width=0.25cm,fill=black},
rvertex/.style = {vertex, fill=red},
bvertex/.style = {vertex, fill=blue},
gvertex/.style = {vertex, fill=green},
nedge/.style = {-}}

\node[vertex, label={[label distance=0.1cm]180:$a$}] (A) at (0, 0) {};
\node[vertex, label={[label distance=0.1cm]180:$b$}] (B) at (0, 2) {};
\node[vertex, label={[label distance=0.1cm]180:$c$}] (C) at (1, 0) {};
\node[vertex, label={[label distance=0.1cm]180:$d$}] (D) at (3, 0) {};
\node[vertex, label={[label distance=0.1cm]180:$e$}] (E) at (2, 2){};

\path [-] (A) edge node {} (B);
\path [-] (C) edge node {} (E);
\path [-] (D) edge node {} (E);


\node[circle, inner sep=3pt] (1) at (7, 0) {$+$};
\node[circle, inner sep=3pt] (2) at (6, 0.666666) {$\oplus$};
\node[circle, inner sep=3pt] (3) at (8, 0.666666) {$\oplus$};
\node[circle, inner sep=3pt] (4) at (5.333333, 1.333333) {$a$};
\node[circle, inner sep=3pt] (5) at (6.666666, 1.333333) {$b$};
\node[circle, inner sep=3pt] (6) at (7.333333, 1.333333) {$+$};
\node[circle, inner sep=3pt] (7) at (8.666666, 1.333333) {$e$};
\node[circle, inner sep=3pt] (8) at (6.666666, 2) {$c$};
\node[circle, inner sep=3pt] (9) at (8, 2) {$d$};

\path [-] (2) edge node {} (1);
\path [-] (1) edge node {} (3);
\path [-] (4) edge node {} (2);
\path [-] (2) edge node {} (5);
\path [-] (6) edge node {} (3);
\path [-] (3) edge node {} (7);
\path [-] (8) edge node {} (6);
\path [-] (6) edge node {} (9);

\end{tikzpicture}

%% file: figs/pf_not_dist.tex
\begin{tikzpicture}[scale=1.0]
\tikzset{vertex/.style={circle, inner sep=0pt, outer sep=0pt, minimum width=0.25cm,fill=black},
rvertex/.style = {vertex, fill=red},
bvertex/.style = {vertex, fill=blue},
gvertex/.style = {vertex, fill=green},
nedge/.style = {-}}

\node[vertex] (a) at (0, 0) {};
\node[vertex] (b) at (-1, 1) {};
\node[vertex] (c) at (0, 1) {};
\node[vertex] (d) at (1, 1) {};

\node[vertex] (f) at (2, 0) {};
\node[vertex] (g) at (2, 1) {};
\node[vertex] (h) at (2, 2) {};

\path [-] (b) edge node {} (a);
\path [-] (c) edge node {} (a);
\path [-] (d) edge node {} (a);
\path [-] (f) edge node {} (g);
\path [-] (g) edge node {} (h);


\node[vertex] (1a) at (4.666666, 0) {};
\node[vertex] (1b) at (4, 1) {};
\node[vertex] (1d) at (5.333333, 1) {};

\node[vertex] (1f) at (6.666666, 0) {};
\node[vertex] (1g) at (6, 1) {};
\node[vertex] (1h) at (7.333333, 1) {};
\node[vertex] (1j) at (6, 2) {};

\path [-] (1b) edge node {} (1a);
\path [-] (1a) edge node {} (1d);
\path [-] (1j) edge node {} (1g);
\path [-] (1g) edge node {} (1f);
\path [-] (1f) edge node {} (1h);

\end{tikzpicture}

%% file: figs/int_not_dist.tex
\begin{tikzpicture}[scale=1.0]
\tikzset{vertex/.style={circle, inner sep=0pt, outer sep=0pt, minimum width=0.25cm,fill=black},
rvertex/.style = {vertex, fill=red},
bvertex/.style = {vertex, fill=blue},
gvertex/.style = {vertex, fill=green},
nedge/.style = {-}}

\node[vertex]  (a) at (0, 0) {};
\node[vertex] (b) at (-1, 1) {};
\node[vertex] (c) at (-0.333333, 1) {};
\node[vertex] (d) at (0.333333, 1) {};
\node[vertex] (e) at (1, 1) {};

\node[vertex] (f) at (2, 0) {};
\node[vertex] (g) at (1.666666, 1) {};
\node[vertex] (h) at (2.333333, 1) {};
\node[vertex] (i) at (1.666666, 2) {};
\node[vertex] (j) at (2.333333, 2) {};

\path [-] (b) edge node {} (a);
\path [-] (c) edge node {} (a);
\path [-] (d) edge node {} (a);
\path [-] (e) edge node {} (a);
\path [-] (i) edge node {} (g);
\path [-] (g) edge node {} (f);
\path [-] (f) edge node {} (h);
\path [-] (h) edge node {} (j);


\node[vertex] (1a) at (4.666666, 0) {};
\node[vertex] (1b) at (4, 1) {};
\node[vertex] (1c) at (4.666666, 1) {};
\node[vertex] (1d) at (5.333333, 1) {};
\node[vertex] (1e) at (6, 1) {};

\node[vertex] (1f) at (6.666666, 0) {};
\node[vertex] (1g) at (6.666666, 1) {};
\node[vertex] (1h) at (7.333333, 1) {};
\node[vertex] (1i) at (4.666666, 2) {};
\node[vertex] (1j) at (6.666666, 2) {};

\path [-] (1b) edge node {} (1a);
\path [-] (1i) edge node {} (1c);
\path [-] (1c) edge node {} (1a);
\path [-] (1a) edge node {} (1d);
\path [-] (1e) edge node {} (1f);
\path [-] (1j) edge node {} (1g);
\path [-] (1g) edge node {} (1f);
\path [-] (1f) edge node {} (1h);

\end{tikzpicture}

%% file: figs/downset_counter.tex
\begin{tikzpicture}[scale=1]
\tikzset{vertex/.style={circle, inner sep=0pt, outer sep=0pt, minimum width=0.25cm,fill=black},
rvertex/.style = {vertex, fill=red},
bvertex/.style = {vertex, fill=blue},
gvertex/.style = {vertex, fill=green},
nedge/.style = {-}}

\node[vertex] (A) at (0, 1) {};
\node[vertex] (B) at (1.5, 0) {};
\node[vertex] (C) at (1, 1) {};
\node[vertex] (D) at (2, 1) {};
\node[vertex] (E) at (1.5, 2) {};

\path [-] (B) edge node {} (C);
\path [-] (B) edge node {} (D);
\path [-] (C) edge node {} (E);
\path [-] (D) edge node {} (E);


\node[vertex] (A1) at (4, 0) {};
\node[vertex] (B1) at (4, 1) {};
\node[vertex] (C1) at (4, 2) {};
\node[vertex] (D1) at (5, 0) {};
\node[vertex] (E1) at (5, 1) {};

\path [-] (A1) edge node {} (B1);
\path [-] (B1) edge node {} (C1);
\path [-] (D1) edge node {} (E1);

\end{tikzpicture}

%% file: figs/p1.tex
\begin{tikzpicture}[scale=0.4]
\tikzset{vertex/.style={circle, inner sep=0pt, outer sep=0pt, minimum width=0.25cm,fill=black},
rvertex/.style = {vertex, fill=red},
bvertex/.style = {vertex, fill=blue},
gvertex/.style = {vertex, fill=green},
nedge/.style = {-}}

\node[vertex] (A) at (0, 0) {};
\node[vertex] (B) at (2, 0) {};
\node[vertex] (C) at (-0.666, 2) {};
\node[vertex] (D) at (1, 2) {};
\node[vertex] (E) at (2.666, 2) {};
\node[vertex] (F) at (0, 4) {};
\node[vertex] (G) at (2, 4) {};

\path [-] (A) edge node {} (C);
\path [-] (A) edge node {} (D);
\path [-] (A) edge node {} (E);
\path [-] (B) edge node {} (E);
\path [-] (D) edge node {} (F);
\path [-] (D) edge node {} (G);
\path [-] (C) edge node {} (G);
\path [-] (E) edge node {} (G);


\node[vertex] (A1) at (7, -1) {};
\node[vertex] (B1) at (9, -1) {};
\node[vertex] (C1) at (6.333, 1) {};
\node[vertex] (D1) at (8, 1) {};
\node[vertex] (E1) at (9.666, 1) {};
\node[vertex] (F1) at (6.333, 3) {};
\node[vertex] (G1) at (8, 5) {};

\path [-] (A1) edge node {} (C1);
\path [-] (A1) edge node {} (D1);
\path [-] (A1) edge node {} (E1);
\path [-] (B1) edge node {} (E1);
\path [-] (B1) edge node {} (G1);
\path [-] (C1) edge node {} (F1);
\path [-] (D1) edge node {} (G1);
\path [-] (F1) edge node {} (G1);

\end{tikzpicture}

%% file: figs/p9.tex
\begin{tikzpicture}[scale=0.4]
\tikzset{vertex/.style={circle, inner sep=0pt, outer sep=0pt, minimum width=0.25cm,fill=black},
rvertex/.style = {vertex, fill=red},
bvertex/.style = {vertex, fill=blue},
gvertex/.style = {vertex, fill=green},
nedge/.style = {-}}

\node[vertex] (A) at (0, 0) {};
\node[vertex] (B) at (2, 0) {};
\node[vertex] (C) at (0, 2) {};
\node[vertex] (D) at (2, 2) {};
\node[vertex] (E) at (1, 4) {};
\node[vertex] (F) at (0, 6) {};
\node[vertex] (G) at (2, 6) {};

\path [-] (A) edge node {} (C);
\path [-] (A) edge node {} (D);
\path [-] (B) edge node {} (D);
\path [-] (B) edge node {} (F);
\path [-] (C) edge node {} (E);
\path [-] (D) edge node {} (G);
\path [-] (E) edge node {} (F);
\path [-] (E) edge node {} (G);


\node[vertex] (A1) at (6, 0) {};
\node[vertex] (B1) at (8, 0) {};
\node[vertex] (C1) at (6, 2) {};
\node[vertex] (D1) at (8, 2) {};
\node[vertex] (E1) at (6, 4) {};
\node[vertex] (F1) at (8, 4) {};
\node[vertex] (G1) at (7, 6) {};

\path [-] (A1) edge node {} (C1);
\path [-] (A1) edge node {} (D1);
\path [-] (B1) edge node {} (D1);
\path [-] (C1) edge node {} (E1);
\path [-] (C1) edge node {} (F1);
\path [-] (D1) edge node {} (F1);
\path [-] (D1) edge node {} (G1);
\path [-] (E1) edge node {} (G1);

\end{tikzpicture}

%% file: figs/p2.tex
\begin{tikzpicture}[scale=0.4]
\tikzset{vertex/.style={circle, inner sep=0pt, outer sep=0pt, minimum width=0.25cm,fill=black},
rvertex/.style = {vertex, fill=red},
bvertex/.style = {vertex, fill=blue},
gvertex/.style = {vertex, fill=green},
nedge/.style = {-}}

\node[vertex] (A) at (0, 0) {};
\node[vertex] (B) at (2, 0) {};
\node[vertex] (C) at (-0.666, 2) {};
\node[vertex] (D) at (1, 2) {};
\node[vertex] (E) at (2.666, 2) {};
\node[vertex] (F) at (0, 4) {};
\node[vertex] (G) at (2, 4) {};

\path [-] (A) edge node {} (D);
\path [-] (A) edge node {} (E);
\path [-] (B) edge node {} (E);
\path [-] (D) edge node {} (F);
\path [-] (D) edge node {} (G);
\path [-] (C) edge node {} (G);
\path [-] (E) edge node {} (G);


\node[vertex] (A1) at (6.333, -1) {};
\node[vertex] (B1) at (9.666, -1) {};
\node[vertex] (C1) at (6.333, 1) {};
\node[vertex] (D1) at (8, -1) {};
\node[vertex] (E1) at (9.666, 1) {};
\node[vertex] (F1) at (6.333, 3) {};
\node[vertex] (G1) at (8, 5) {};

\path [-] (A1) edge node {} (C1);
\path [-] (A1) edge node {} (E1);
\path [-] (B1) edge node {} (E1);
\path [-] (B1) edge node {} (G1);
\path [-] (C1) edge node {} (F1);
\path [-] (D1) edge node {} (G1);
\path [-] (F1) edge node {} (G1);

\end{tikzpicture}

%% file: figs/p10.tex
\begin{tikzpicture}[scale=0.4]
\tikzset{vertex/.style={circle, inner sep=0pt, outer sep=0pt, minimum width=0.25cm,fill=black},
rvertex/.style = {vertex, fill=red},
bvertex/.style = {vertex, fill=blue},
gvertex/.style = {vertex, fill=green},
nedge/.style = {-}}

\node[vertex] (A) at (0, 0) {};
\node[vertex] (B) at (2, 0) {};
\node[vertex] (C) at (0, 2) {};
\node[vertex] (D) at (2, 2) {};
\node[vertex] (E) at (0, 4) {};
\node[vertex] (F) at (2, 4) {};
\node[vertex] (G) at (1, 6) {};

\path [-] (A) edge node {} (C);
\path [-] (A) edge node {} (D);
\path [-] (B) edge node {} (D);
\path [-] (B) edge node {} (G);
\path [-] (C) edge node {} (E);
\path [-] (C) edge node {} (F);
\path [-] (D) edge node {} (E);
\path [-] (D) edge node {} (F);
\path [-] (E) edge node {} (G);


\node[vertex] (A1) at (6, 0) {};
\node[vertex] (B1) at (8, 0) {};
\node[vertex] (C1) at (6, 2) {};
\node[vertex] (D1) at (8, 2) {};
\node[vertex] (E1) at (6, 4) {};
\node[vertex] (F1) at (8, 4) {};
\node[vertex] (G1) at (7, 6) {};

\path [-] (A1) edge node {} (C1);
\path [-] (A1) edge node {} (D1);
\path [-] (B1) edge node {} (E1);
\path [-] (C1) edge node {} (E1);
\path [-] (C1) edge node {} (F1);
\path [-] (D1) edge node {} (F1);
\path [-] (D1) edge node {} (G1);
\path [-] (E1) edge node {} (G1);

\end{tikzpicture}

%% file: figs/p3.tex
\begin{tikzpicture}[scale=0.4]
\tikzset{vertex/.style={circle, inner sep=0pt, outer sep=0pt, minimum width=0.25cm,fill=black},
rvertex/.style = {vertex, fill=red},
bvertex/.style = {vertex, fill=blue},
gvertex/.style = {vertex, fill=green},
nedge/.style = {-}}

\node[vertex] (A) at (-0.666, 0) {};
\node[vertex] (B) at (1, 0) {};
\node[vertex] (C) at (2.666, 0) {};
\node[vertex] (D) at (0, 2) {};
\node[vertex] (E) at (2, 2) {};
\node[vertex] (F) at (0, 4) {};
\node[vertex] (G) at (2, 4) {};

\path [-] (A) edge node {} (D);
\path [-] (A) edge node {} (E);
\path [-] (B) edge node {} (D);
\path [-] (B) edge node {} (E);
\path [-] (C) edge node {} (E);
\path [-] (D) edge node {} (F);
\path [-] (D) edge node {} (G);
\path [-] (E) edge node {} (G);


\node[vertex] (A1) at (6.333, -1) {};
\node[vertex] (B1) at (8, -1) {};
\node[vertex] (C1) at (9.666, -1) {};
\node[vertex] (D1) at (7, 1) {};
\node[vertex] (E1) at (9, 1) {};
\node[vertex] (F1) at (8, 3) {};
\node[vertex] (G1) at (8, 5) {};

\path [-] (A1) edge node {} (D1);
\path [-] (A1) edge node {} (E1);
\path [-] (B1) edge node {} (D1);
\path [-] (B1) edge node {} (E1);
\path [-] (C1) edge node {} (E1);
\path [-] (C1) edge node {} (G1);
\path [-] (D1) edge node {} (F1);
\path [-] (F1) edge node {} (G1);

\end{tikzpicture}

%% file: figs/p11.tex
\begin{tikzpicture}[scale=0.4]
\tikzset{vertex/.style={circle, inner sep=0pt, outer sep=0pt, minimum width=0.25cm,fill=black},
rvertex/.style = {vertex, fill=red},
bvertex/.style = {vertex, fill=blue},
gvertex/.style = {vertex, fill=green},
nedge/.style = {-}}

\node[vertex] (A) at (0, 0) {};
\node[vertex] (B) at (2, 0) {};
\node[vertex] (C) at (0, 2) {};
\node[vertex] (D) at (1.333, 2) {};
\node[vertex] (E) at (0, 4) {};
\node[vertex] (F) at (2, 4) {};
\node[vertex] (G) at (1, 6) {};

\path [-] (A) edge node {} (C);
\path [-] (A) edge node {} (D);
\path [-] (B) edge node {} (D);
\path [-] (B) edge node {} (F);
\path [-] (C) edge node {} (E);
\path [-] (C) edge node {} (F);
\path [-] (E) edge node {} (G);
\path [-] (F) edge node {} (G);


\node[vertex] (A1) at (6, 0) {};
\node[vertex] (B1) at (8, 0) {};
\node[vertex] (C1) at (6, 2) {};
\node[vertex] (D1) at (8, 2) {};
\node[vertex] (E1) at (6, 4) {};
\node[vertex] (F1) at (8, 4) {};
\node[vertex] (G1) at (7, 6) {};

\path [-] (A1) edge node {} (C1);
\path [-] (A1) edge node {} (D1);
\path [-] (B1) edge node {} (D1);
\path [-] (B1) edge node {} (E1);
\path [-] (C1) edge node {} (E1);
\path [-] (C1) edge node {} (F1);
\path [-] (D1) edge node {} (G1);
\path [-] (E1) edge node {} (G1);

\end{tikzpicture}

%% file: figs/p4.tex
\begin{tikzpicture}[scale=0.4]
\tikzset{vertex/.style={circle, inner sep=0pt, outer sep=0pt, minimum width=0.25cm,fill=black},
rvertex/.style = {vertex, fill=red},
bvertex/.style = {vertex, fill=blue},
gvertex/.style = {vertex, fill=green},
nedge/.style = {-}}

\node[vertex] (A) at (0, 0) {};
\node[vertex] (B) at (2, 0) {};
\node[vertex] (C) at (0, 2) {};
\node[vertex] (D) at (2, 2) {};
\node[vertex] (E) at (0, 4) {};
\node[vertex] (F) at (2, 4) {};
\node[vertex] (G) at (1, 6) {};

\path [-] (A) edge node {} (C);
\path [-] (A) edge node {} (D);
\path [-] (B) edge node {} (D);
\path [-] (C) edge node {} (E);
\path [-] (C) edge node {} (F);
\path [-] (D) edge node {} (F);
\path [-] (E) edge node {} (G);
\path [-] (F) edge node {} (G);


\node[vertex] (A1) at (6, -1) {};
\node[vertex] (B1) at (8, -1) {};
\node[vertex] (C1) at (6, 1) {};
\node[vertex] (D1) at (8, 1) {};
\node[vertex] (E1) at (7, 3) {};
\node[vertex] (F1) at (7, 5) {};
\node[vertex] (G1) at (7, 7) {};

\path [-] (A1) edge node {} (C1);
\path [-] (A1) edge node {} (D1);
\path [-] (B1) edge node {} (D1);
\path [-] (B1) edge node {} (F1);
\path [-] (C1) edge node {} (E1);
\path [-] (D1) edge node {} (G1);
\path [-] (E1) edge node {} (F1);
\path [-] (F1) edge node {} (G1);

\end{tikzpicture}

%% file: figs/p12.tex
\begin{tikzpicture}[scale=0.4]
\tikzset{vertex/.style={circle, inner sep=0pt, outer sep=0pt, minimum width=0.25cm,fill=black},
rvertex/.style = {vertex, fill=red},
bvertex/.style = {vertex, fill=blue},
gvertex/.style = {vertex, fill=green},
nedge/.style = {-}}

\node[vertex] (A) at (0, 0) {};
\node[vertex] (B) at (2, 0) {};
\node[vertex] (C) at (0, 2) {};
\node[vertex] (D) at (2, 2) {};
\node[vertex] (E) at (0, 4) {};
\node[vertex] (F) at (2, 4) {};
\node[vertex] (G) at (1, 6) {};

\path [-] (A) edge node {} (C);
\path [-] (A) edge node {} (D);
\path [-] (B) edge node {} (D);
\path [-] (C) edge node {} (E);
\path [-] (C) edge node {} (F);
\path [-] (D) edge node {} (F);
\path [-] (F) edge node {} (G);


\node[vertex] (A1) at (6, -1) {};
\node[vertex] (B1) at (8, -1) {};
\node[vertex] (C1) at (6, 1) {};
\node[vertex] (D1) at (8, 1) {};
\node[vertex] (E1) at (7, 3) {};
\node[vertex] (F1) at (7, 5) {};
\node[vertex] (G1) at (7, 7) {};

\path [-] (A1) edge node {} (C1);
\path [-] (A1) edge node {} (D1);
\path [-] (B1) edge node {} (D1);
\path [-] (B1) edge node {} (F1);
\path [-] (C1) edge node {} (E1);
\path [-] (E1) edge node {} (F1);
\path [-] (F1) edge node {} (G1);

\end{tikzpicture}

%% file: figs/p5.tex
\begin{tikzpicture}[scale=0.4]
\tikzset{vertex/.style={circle, inner sep=0pt, outer sep=0pt, minimum width=0.25cm,fill=black},
rvertex/.style = {vertex, fill=red},
bvertex/.style = {vertex, fill=blue},
gvertex/.style = {vertex, fill=green},
nedge/.style = {-}}

\node[vertex] (A) at (-0.666, 0) {};
\node[vertex] (B) at (1, 0) {};
\node[vertex] (C) at (2.666, 0) {};
\node[vertex] (D) at (0, 2) {};
\node[vertex] (E) at (2, 2) {};
\node[vertex] (F) at (0, 4) {};
\node[vertex] (G) at (2, 4) {};

\path [-] (A) edge node {} (F);
\path [-] (A) edge node {} (D);
\path [-] (B) edge node {} (D);
\path [-] (C) edge node {} (D);
\path [-] (C) edge node {} (E);
\path [-] (D) edge node {} (G);
\path [-] (E) edge node {} (F);
\path [-] (E) edge node {} (G);


\node[vertex] (A1) at (6.333, -1) {};
\node[vertex] (B1) at (8, -1) {};
\node[vertex] (C1) at (9.666, -1) {};
\node[vertex] (D1) at (7, 1) {};
\node[vertex] (E1) at (9, 1) {};
\node[vertex] (F1) at (8, 3) {};
\node[vertex] (G1) at (8, 5) {};

\path [-] (A1) edge node {} (D1);
\path [-] (A1) edge node {} (E1);
\path [-] (B1) edge node {} (F1);
\path [-] (B1) edge node {} (E1);
\path [-] (C1) edge node {} (E1);
\path [-] (C1) edge node {} (G1);
\path [-] (D1) edge node {} (F1);
\path [-] (F1) edge node {} (G1);

\end{tikzpicture}

%% file: figs/p13.tex
\begin{tikzpicture}[scale=0.4]
\tikzset{vertex/.style={circle, inner sep=0pt, outer sep=0pt, minimum width=0.25cm,fill=black},
rvertex/.style = {vertex, fill=red},
bvertex/.style = {vertex, fill=blue},
gvertex/.style = {vertex, fill=green},
nedge/.style = {-}}

\node[vertex] (A) at (0, 0) {};
\node[vertex] (B) at (2, 0) {};
\node[vertex] (C) at (0, 2) {};
\node[vertex] (D) at (2, 2) {};
\node[vertex] (E) at (-0.666, 4) {};
\node[vertex] (F) at (1, 4) {};
\node[vertex] (G) at (2.666, 4) {};

\path [-] (A) edge node {} (C);
\path [-] (A) edge node {} (D);
\path [-] (B) edge node {} (D);
\path [-] (C) edge node {} (E);
\path [-] (C) edge node {} (F);
\path [-] (C) edge node {} (G);
\path [-] (D) edge node {} (F);
\path [-] (D) edge node {} (G);


\node[vertex] (A1) at (6, -1) {};
\node[vertex] (B1) at (8, -1) {};
\node[vertex] (C1) at (6, 1) {};
\node[vertex] (D1) at (7.666, 1) {};
\node[vertex] (E1) at (7, 3) {};
\node[vertex] (F1) at (6, 5) {};
\node[vertex] (G1) at (8, 5) {};

\path [-] (A1) edge node {} (C1);
\path [-] (A1) edge node {} (D1);
\path [-] (B1) edge node {} (D1);
\path [-] (B1) edge node {} (F1);
\path [-] (B1) edge node {} (G1);
\path [-] (C1) edge node {} (E1);
\path [-] (E1) edge node {} (F1);
\path [-] (E1) edge node {} (G1);

\end{tikzpicture}

%% file: figs/p6.tex
\begin{tikzpicture}[scale=0.4]
\tikzset{vertex/.style={circle, inner sep=0pt, outer sep=0pt, minimum width=0.25cm,fill=black},
rvertex/.style = {vertex, fill=red},
bvertex/.style = {vertex, fill=blue},
gvertex/.style = {vertex, fill=green},
nedge/.style = {-}}

\node[vertex] (A) at (1, 0) {};
\node[vertex] (B) at (0, 2) {};
\node[vertex] (C) at (2, 2) {};
\node[vertex] (D) at (0, 4) {};
\node[vertex] (E) at (2, 4) {};
\node[vertex] (F) at (0, 6) {};
\node[vertex] (G) at (2, 6) {};

\path [-] (A) edge node {} (B);
\path [-] (A) edge node {} (C);
\path [-] (B) edge node {} (D);
\path [-] (C) edge node {} (D);
\path [-] (C) edge node {} (E);
\path [-] (D) edge node {} (G);
\path [-] (E) edge node {} (F);
\path [-] (E) edge node {} (G);


\node[vertex] (A1) at (7, -1) {};
\node[vertex] (B1) at (6, 1) {};
\node[vertex] (C1) at (8, 1) {};
\node[vertex] (D1) at (6, 3) {};
\node[vertex] (E1) at (8, 3) {};
\node[vertex] (F1) at (7, 5) {};
\node[vertex] (G1) at (7, 7) {};

\path [-] (A1) edge node {} (B1);
\path [-] (A1) edge node {} (C1);
\path [-] (B1) edge node {} (D1);
\path [-] (B1) edge node {} (E1);
\path [-] (C1) edge node {} (E1);
\path [-] (C1) edge node {} (G1);
\path [-] (D1) edge node {} (F1);
\path [-] (F1) edge node {} (G1);

\end{tikzpicture}

%% file: figs/p14.tex
\begin{tikzpicture}[scale=0.4]
\tikzset{vertex/.style={circle, inner sep=0pt, outer sep=0pt, minimum width=0.25cm,fill=black},
rvertex/.style = {vertex, fill=red},
bvertex/.style = {vertex, fill=blue},
gvertex/.style = {vertex, fill=green},
nedge/.style = {-}}

\node[vertex] (A) at (0, 0) {};
\node[vertex] (B) at (2, 0) {};
\node[vertex] (C) at (0, 2) {};
\node[vertex] (D) at (2, 2) {};
\node[vertex] (E) at (-0.666, 4) {};
\node[vertex] (F) at (1, 4) {};
\node[vertex] (G) at (2.666, 4) {};

\path [-] (A) edge node {} (C);
\path [-] (A) edge node {} (F);
\path [-] (B) edge node {} (C);
\path [-] (B) edge node {} (D);
\path [-] (C) edge node {} (G);
\path [-] (D) edge node {} (E);
\path [-] (D) edge node {} (F);
\path [-] (D) edge node {} (G);


\node[vertex] (A1) at (6, -1) {};
\node[vertex] (B1) at (8, -1) {};
\node[vertex] (C1) at (6, 1) {};
\node[vertex] (D1) at (7.5, 1) {};
\node[vertex] (E1) at (6, 3) {};
\node[vertex] (F1) at (8, 3) {};
\node[vertex] (G1) at (7, 5) {};

\path [-] (A1) edge node {} (C1);
\path [-] (A1) edge node {} (D1);
\path [-] (B1) edge node {} (D1);
\path [-] (B1) edge node {} (F1);
\path [-] (B1) edge node {} (G1);
\path [-] (C1) edge node {} (E1);
\path [-] (C1) edge node {} (F1);
\path [-] (E1) edge node {} (G1);

\end{tikzpicture}

%% file: figs/p7.tex
\begin{tikzpicture}[scale=0.4]
\tikzset{vertex/.style={circle, inner sep=0pt, outer sep=0pt, minimum width=0.25cm,fill=black},
rvertex/.style = {vertex, fill=red},
bvertex/.style = {vertex, fill=blue},
gvertex/.style = {vertex, fill=green},
nedge/.style = {-}}

\node[vertex] (A) at (1, 0) {};
\node[vertex] (B) at (0, 2) {};
\node[vertex] (C) at (2, 2) {};
\node[vertex] (D) at (0, 4) {};
\node[vertex] (E) at (2, 4) {};
\node[vertex] (F) at (0, 6) {};
\node[vertex] (G) at (2, 6) {};

\path [-] (A) edge node {} (C);
\path [-] (B) edge node {} (D);
\path [-] (C) edge node {} (D);
\path [-] (C) edge node {} (E);
\path [-] (D) edge node {} (G);
\path [-] (E) edge node {} (F);
\path [-] (E) edge node {} (G);


\node[vertex] (A1) at (6, -1) {};
\node[vertex] (B1) at (8, -1) {};
\node[vertex] (C1) at (7, 1) {};
\node[vertex] (D1) at (6, 3) {};
\node[vertex] (E1) at (8, 3) {};
\node[vertex] (F1) at (7, 5) {};
\node[vertex] (G1) at (7, 7) {};

\path [-] (A1) edge node {} (C1);
\path [-] (B1) edge node {} (E1);
\path [-] (B1) edge node {} (G1);
\path [-] (C1) edge node {} (D1);
\path [-] (C1) edge node {} (E1);
\path [-] (D1) edge node {} (F1);
\path [-] (F1) edge node {} (G1);

\end{tikzpicture}

%% file: figs/p15.tex
\begin{tikzpicture}[scale=0.4]
\tikzset{vertex/.style={circle, inner sep=0pt, outer sep=0pt, minimum width=0.25cm,fill=black},
rvertex/.style = {vertex, fill=red},
bvertex/.style = {vertex, fill=blue},
gvertex/.style = {vertex, fill=green},
nedge/.style = {-}}

\node[vertex] (A) at (-0.666, 0) {};
\node[vertex] (B) at (1, 0) {};
\node[vertex] (C) at (2.666, 0) {};
\node[vertex] (D) at (0, 2) {};
\node[vertex] (E) at (2, 2) {};
\node[vertex] (F) at (0, 4) {};
\node[vertex] (G) at (2, 4) {};

\path [-] (B) edge node {} (D);
\path [-] (B) edge node {} (E);
\path [-] (C) edge node {} (E);
\path [-] (D) edge node {} (F);
\path [-] (D) edge node {} (G);
\path [-] (E) edge node {} (G);


\node[vertex] (A1) at (5.333, -1) {};
\node[vertex] (B1) at (7, -1) {};
\node[vertex] (C1) at (8.666, -1) {};
\node[vertex] (D1) at (6, 1) {};
\node[vertex] (E1) at (8, 1) {};
\node[vertex] (F1) at (6.5, 3) {};
\node[vertex] (G1) at (7, 5) {};

\path [-] (A1) edge node {} (D1);
\path [-] (A1) edge node {} (E1);
\path [-] (B1) edge node {} (E1);
\path [-] (B1) edge node {} (G1);
\path [-] (D1) edge node {} (F1);
\path [-] (F1) edge node {} (G1);

\end{tikzpicture}

%% file: figs/p8.tex
\begin{tikzpicture}[scale=0.4]
\tikzset{vertex/.style={circle, inner sep=0pt, outer sep=0pt, minimum width=0.25cm,fill=black},
rvertex/.style = {vertex, fill=red},
bvertex/.style = {vertex, fill=blue},
gvertex/.style = {vertex, fill=green},
nedge/.style = {-}}

\node[vertex] (A) at (0, 0) {};
\node[vertex] (B) at (2, 0) {};
\node[vertex] (C) at (0, 2) {};
\node[vertex] (D) at (2, 2) {};
\node[vertex] (E) at (0, 4) {};
\node[vertex] (F) at (2, 4) {};
\node[vertex] (G) at (1, 6) {};

\path [-] (A) edge node {} (C);
\path [-] (A) edge node {} (D);
\path [-] (B) edge node {} (D);
\path [-] (B) edge node {} (E);
\path [-] (C) edge node {} (E);
\path [-] (D) edge node {} (F);
\path [-] (D) edge node {} (G);
\path [-] (E) edge node {} (G);


\node[vertex] (A1) at (6, 0) {};
\node[vertex] (B1) at (8, 0) {};
\node[vertex] (C1) at (6, 2) {};
\node[vertex] (D1) at (8, 2) {};
\node[vertex] (E1) at (6, 4) {};
\node[vertex] (F1) at (8, 4) {};
\node[vertex] (G1) at (7, 6) {};

\path [-] (A1) edge node {} (C1);
\path [-] (A1) edge node {} (F1);
\path [-] (B1) edge node {} (D1);
\path [-] (B1) edge node {} (C1);
\path [-] (C1) edge node {} (E1);
\path [-] (D1) edge node {} (F1);
\path [-] (D1) edge node {} (G1);
\path [-] (E1) edge node {} (G1);

\end{tikzpicture}

%% file: figs/p16.tex
\begin{tikzpicture}[scale=0.4]
\tikzset{vertex/.style={circle, inner sep=0pt, outer sep=0pt, minimum width=0.25cm,fill=black},
rvertex/.style = {vertex, fill=red},
bvertex/.style = {vertex, fill=blue},
gvertex/.style = {vertex, fill=green},
nedge/.style = {-}}

\node[vertex] (A) at (0, 0) {};
\node[vertex] (B) at (2, 0) {};
\node[vertex] (C) at (-0.666, 2) {};
\node[vertex] (D) at (1, 2) {};
\node[vertex] (E) at (2.666, 2) {};
\node[vertex] (F) at (0, 4) {};
\node[vertex] (G) at (2, 4) {};

\path [-] (A) edge node {} (C);
\path [-] (A) edge node {} (D);
\path [-] (A) edge node {} (E);
\path [-] (B) edge node {} (E);
\path [-] (C) edge node {} (F);
\path [-] (C) edge node {} (G);
\path [-] (E) edge node {} (G);


\node[vertex] (A1) at (7, -1) {};
\node[vertex] (B1) at (9, -1) {};
\node[vertex] (C1) at (6.333, 1) {};
\node[vertex] (D1) at (8, 1) {};
\node[vertex] (E1) at (9.666, 1) {};
\node[vertex] (F1) at (8, 3) {};
\node[vertex] (G1) at (8, 5) {};

\path [-] (A1) edge node {} (C1);
\path [-] (A1) edge node {} (D1);
\path [-] (A1) edge node {} (E1);
\path [-] (B1) edge node {} (E1);
\path [-] (B1) edge node {} (G1);
\path [-] (C1) edge node {} (F1);
\path [-] (F1) edge node {} (G1);

\end{tikzpicture}